\documentclass[11pt]{article}
\usepackage{amsmath,amssymb,amsthm}

\newtheorem{theorem}{Theorem}[section]
\newtheorem{lemma}[theorem]{Lemma}
\newtheorem{proposition}[theorem]{Proposition}

\theoremstyle{definition}
\newtheorem{definition}[theorem]{Definition}

\numberwithin{equation}{section}

\renewcommand{\labelenumi}{\textup{(\theenumi)}}

\renewcommand{\phi}{\varphi}

\newcommand{\Homeo}{\operatorname{Homeo}}
\newcommand{\id}{\operatorname{id}}
\newcommand{\Ker}{\operatorname{Ker}}
\newcommand{\orb}{\operatorname{orb}}

\newcommand{\N}{\mathbb{N}}
\newcommand{\Z}{\mathbb{Z}}

\title{Full groups of  Cuntz--Krieger algebras
and Higman--Thompson groups}

\author{Kengo Matsumoto \\
Department of Mathematics \\
Joetsu University of Education \\
Joetsu, Niigata 943-8512, Japan
\and
Hiroki Matui \\
Graduate School of Science \\
Chiba University \\
Inage-ku, Chiba 263-8522, Japan}
\date{}
\renewcommand{\labelenumi}{\textup{(\theenumi)}}

\begin{document}
\maketitle

\begin{abstract}
In this paper,
we will study representations of the continuous full group $\Gamma_A$ 
of a one-sided topological Markov shift $(X_A,\sigma_A)$
for an irreducible matrix $A$ with entries in 
$\{0,1\}$ as a generalization of Higman--Thompson groups
$V_N, 1<N \in {\mathbb{N}}$.
We will show that 
the group
$\Gamma_A$ can be represented as
a group $\Gamma_A^{\operatorname{tab}}$ of
matrices, called $A$-adic tables,
with entries in admissible words of the shift space $X_A$,
and 
 a group $\Gamma_A^{\operatorname{PL}}$
of right continuous piecewise linear functions, 
called $A$-adic PL functions, on $[0,1]$
with finite singularities. 
\end{abstract}

\def\Zp{{ {\mathbb{Z}}_+ }}
\def\Z{{ {\mathbb{Z}}}}
\def\N{{ {\mathbb{N}}}}
\def\Ext{{{\operatorname{Ext}}}}
\def\Im{{{\operatorname{Im}}}}
\def\Max{{{\operatorname{Max}}}}
\def\Min{{{\operatorname{Min}}}}
\def\max{{{\operatorname{max}}}}
\def\min{{{\operatorname{min}}}}
\def\Aut{{{\operatorname{Aut}}}}
\def\Ad{{{\operatorname{Ad}}}}
\def\Ker{{{\operatorname{Ker}}}}
\def\dim{{{\operatorname{dim}}}}
\def\det{{{\operatorname{det}}}}
\def\orb{{{\operatorname{orb}}}}
\def\supp{{{\operatorname{supp}}}}
\def\id{{{\operatorname{id}}}}
\def\tab{{\operatorname{tab}}}
\def\PL{{\operatorname{PL}}}

\def\OB{{ {\mathcal{O}}_B}}
\def\OA{{ {\mathcal{O}}_A}}
\def\A{{ {\mathcal{A}}}}
\def\B{{ {\mathcal{B}}}}
\def\C{{ {\mathcal{C}}}}
\def\E{{ {\mathcal{E}}}}
\def\T{{ {\mathcal{T}}}}
\def\S{{ {\mathcal{S}}}}
\def\U{{ {\mathcal{U}}}}
\def\V{{ {\mathcal{V}}}}
\def\TG{{ {\frak T}_G }}
\def\AA{{ {\mathcal{A}}_A}}
\def\DB{{ {\mathcal{D}}_B}}
\def\DA{{ {\mathcal{D}}_A }}
\def\FA{{ {\mathcal{F}}_A }}



Primary  20F38, 37A55;
Secondary  20F65, 37B10, 46L35

\section{Introduction}
\bigskip

In 1960's, R. J. Thompson
has initiated a study of finitely presented simple infinite groups. 
He has discovered first two such groups in \cite{Thompson}.
They are now known as the groups $V_2$ and $T_2$. 
G. Higman has generalized the group $V_2$  to
infinite family of finitely presented infinite groups.
One of such families are groups written  
$V_N, 1<N \in {\mathbb{N}}$ which are called the Higman--Thompson groups.
They are  finitely presented 
and their commutator subgroups
are simple.
Their abelianizations are trivial if $N$ is even, 
and ${\mathbb{Z}}_2$ if $N$ is odd.
K. S.  Brown has extended the groups $V_N$
to triplets of infinite families
$F_N \subset T_N \subset V_N, 1<N \in {\mathbb{N}}$,
and proved that each of the groups is finitely presented
(\cite{Brown}).
The Higman--Thompson group 
$V_N$ 
is known to 
be represented as the group of right continuous piecewise linear functions 
$f: [0,1) \longrightarrow [0,1)$
having finitely many singularities such that
all singularities of $f$ are in ${\mathbb{Z}}[\frac{1}{N}]$,
the derivative of $f$ at any non-singular point is
$N^k$ for some $k \in {\mathbb{Z}}$ and 
$f({\mathbb{Z}}[\frac{1}{N}] \cap [0,1)) 
= {\mathbb{Z}}[\frac{1}{N}] \cap [0,1)$
(\cite{Thompson}).
See \cite{CFP} for general reference on these groups.

 V. Nekrashevych \cite{Nek} has shown that 
the Higman--Thompson group $V_N$ appears as  a certain subgroup
of the unitary group of the Cuntz algebra ${\mathcal{O}}_N$.
The second named author has observed in 
\cite[Remark 6.3]{Matui2015}
that the subgroup is nothing but 
the continuous full group $\Gamma_N$
of ${\mathcal{O}}_N$, 
which is also realized as the topological full group of 
the associated groupoid.
Such full groups have arisen  from a study of orbit equivalence of symbolic dynamics (\cite{MaPacific}).

Recently the authors have studied full groups of 
the Cuntz--Krieger algebras 
and full groups of the groupoids coming from shifts of finite type.
The first named author has studied 
the normalizer groups of the canonical maximal abelian $C^*$-subalgebras
in the Cuntz--Krieger algebras 
which are called the continuous full groups from the view point of orbit equivalences of topological Markov shifts and classification of $C^*$-algebras 
(\cite{MaPacific}, \cite{MaPAMS}, etc.), 
and showed that the continuous full groups are complete invariants for 
the continuous orbit equivalence classes of the underlying topological Markov shifts
(\cite{MaPre2012}, more generally \cite{Matui2015}).
The second named author 
has studied the continuous full groups of more general  
\'{e}tale groupoids (\cite{MatuiPLMS},
\cite{Matui2013}, \cite{Matui2015}, etc.),
and called them the topological full groups of \'{e}tale groupoids.
He has proved that if 
an \'{e}tale groupoid is minimal, 
the topological full group of the groupoid is a complete invariant 
for the isomorphism  class of the groupoid.
He has also shown that if a groupoid comes from a shift of finite type,
the topological full group is of type $F_\infty$ and in particular finitely presented. 
He has furthermore obtained that 
the topological full groups for shifts of finite type 
are simple if and only if 
its homology group $H_0(G_A)$ of the groupoid $G_A$
is $2$-divisible,
and that its commutator subgroups are always simple.
We have obtained  the following  results
on the group $\Gamma_A$ for the topological Markov shift 
$(X_A,\sigma_A)$ defined by an irreducible  square
matrix with entries in $\{0,1\}$.  
\begin{theorem}[\cite{MaPre2012},
\cite{MatsumotoMatui2014}, 
\cite{Matui2015}] \label{thm:1.1}
Let $A$ and $B$ be irreducible, not any permutation matrices with entries in $\{0,1\}$. 
The following conditions are equivalent: 
\begin{enumerate}
\item The one-sided topological Markov shifts
$(X_A,\sigma_A)$ and $(X_B,\sigma_B)$ are 
continuously orbit equivalent. 
\item The \'etale groupoids $G_A$ and $G_B$ are isomorphic. 
\item The groups $\Gamma_A$ and $\Gamma_B$ are isomorphic.
\item The Cuntz--Krieger algebras $\OA$ and $\OB$ are isomorphic
and
$\det(\id-A) = \det(\id-B)$. 
\end{enumerate}
\end{theorem}
Suppose that 
$A$ is an $N \times N$ matrix and 
$B$ is an $M \times M$ matrix.
It is well-known that 
the Cuntz--Krieger algebras $\OA$ and $\OB$ are isomorphic
if and only if there exists an isomorphism $\Phi$ of groups from
$\Z^N /(\id - A^t)\Z^N$ to
$\Z^M /(\id - B^t)\Z^M$ 
such that 
$\Phi(u_A) = u_B$
where
$u_A$ and $u_B$ are the classes of the vectors
$[1,\dots,1]$ (\cite{Ro}).
Hence the isomorphism classes of the groups $\Gamma_A$
are completely classified in terms of the underlying matrices
$A$, so that 
there exist  an infinite family of finitely presented infinite simple groups 
of the form $\Gamma_A$.

In this paper,
we will study representations of the group $\Gamma_A$ for an irreducible matrix $A$ with entries in 
$\{0,1\}$ as a generalization of the  Higman--Thompson groups
$V_N, 1<N \in \N$.
The group $\Gamma_A$ has been originally defined as the group of homeomorphisms $\tau$
on the shift space $X_A$ of a topological Markov shift 
$(X_A,\sigma)$ such that 
\begin{equation}
\sigma_A^{k_\tau(x)}(\tau(x)) = \sigma_A^{l_\tau(x)}(x), \qquad x \in X_A 
\label{eq:tau1}
\end{equation}
for some continuous functions $k_\tau, l_\tau: X_A \longrightarrow \Zp$
(it is written $[\sigma_A]$ in the earlier papers \cite{MaPacific}, \cite{MaDCDS}).
If the matrix $A$ is the $N \times N$-matrix whose entries 
are all $1$'s,
the group $\Gamma_A$ coincides with the Higman--Thompson group
$V_N$ of order $N$.

We will introduce a notion of $A$-adic PL (piecewise linear) function 
which is a right continuous bijective piecewise linear function on the interval $[0,1)$
associated with the matrix $A$
to represent an element of the group $\Gamma_A$.
Let $1<\beta\in \mathbb{R}$ be the Perron--Frobenius eigenvalue of $A$.
Let us denote by
$\Z[\frac{1}{\beta}, \beta]$ the set of $\beta$-adic rationals which is defined by
$$
\Z[\frac{1}{\beta}, \beta] = \{ \frac{a_0 + a_1\beta+ a_2 \beta^2+\cdots + a_n \beta^n}{\beta^n}
\mid a_0, a_1, \dots, a_n \in \Z\}
$$  
Then the group of $A$-adic PL functions on $[0,1)$ 
is realized as a subgroup of right continuous bijective piecewise linear functions
$f$ on $[0,1) $ 
having finitely many singularities such that
all singularities of $f$ are in $\Z[\frac{1}{\beta}, \beta],$
the derivative of $f$ at any non-singular point is
$\beta^k$ for some $k \in {\mathbb{Z}}$ and 
$f(\Z[\frac{1}{\beta}, \beta] \cap [0,1)) \subset \Z[\frac{1}{\beta}, \beta] \cap [0,1)$.
See Section 4 for the precise definition.
We also introduce 
a notion of $A$-adic table in order to  represent elements of $\Gamma_A$
which  is a matrix
\begin{equation*}
\begin{bmatrix}
\mu(1) & \mu(2) & \cdots & \mu(m) \\
\nu(1) & \nu(2) & \cdots & \nu(m) 
\end{bmatrix}
\end{equation*}
with entries in admissible words  
$\nu(i), \mu(i),  i=1,\dots, m$
of the one-sided topological Markov shift 
$(X_A,\sigma_A)$
satisfying certain properties.
We may define an equivalence relation of the $A$-adic tables,
and a product structure in  the set $\Gamma_A^{\tab}$
of the equivalence classes of $A$-adic tables
which makes it a group.
We will show the following theorem
which is a generalization of a well-known result for 
the Higman--Thompson groups.
Assume that $A$ is an irreducible and non permutation matrix with entries in 
$\{0,1\}$.
\begin{theorem}[Theorem \ref{thm:SFTPL}]
There exist canonical isomorphisms of discrete groups among 
the continuous full group $\Gamma_A$, 
the group $\Gamma^{\tab}_A$ of the equivalence classes  of $A$-adic tables,
and
the group $\Gamma^{\PL}_A$ of $A$-adic PL functions on $[0,1)$,
that is 
\begin{equation*}
\Gamma_A  \cong \Gamma^{\tab}_A \cong \Gamma^{\PL}_A.
\end{equation*}
\end{theorem}
Let $1 < \beta \in {\mathbb{R}}$ 
be the Perron--Frobenius eigenvalue of $A$.
For $\tau \in \Gamma_A$,
we put
$d_\tau(x) = l_\tau(x) - k_\tau(x)$, $x \in X_A$
for the continuous functions $k_\tau, l_\tau$ satisfying \eqref{eq:tau1}.
We define the derivative $D_\tau$ of $\tau$ 
as a real valued continuous function on $X_A$:
\begin{equation*}
D_\tau(x) = \beta^{d_{\tau}(x)}, \qquad x \in X_A.
\end{equation*}
We know that $D_\tau$ satisfies the following law  of derivative:
\begin{equation*}
 D_{\tau_2\circ \tau_1} = D_{\tau_1}\cdot (D_{\tau_2} \circ \tau_1),
 \qquad 
 D_{\tau^{-1}} = (D_{\tau} \circ \tau^{-1})^{-1}
 \end{equation*} 
for $\tau, \tau_1, \tau_2 \in \Gamma_A$ 
(Proposition \ref{prop:derivative}).

The continuous full group
$\Gamma_A$
is isomorphic to 
the group $\Gamma^{\PL}_A$ 
of all $A$-adic PL functions on $[0,1)$
by the above theorem.
We will show that $\tau \in \Gamma_A$
is realized as an $A$-adic PL function on $[0,1)$
in the following way, 
where $X_A$ is endowed with lexicographic order.
\begin{theorem}[Theorem \ref{thm:rhoA}]
There exists an order preserving continuous surjection
$\rho_A: X_A \longrightarrow [0,1]$
from the shift space $X_A$ of a one-sided topological Markov shift 
$(X_A,\sigma_A)$ to the closed interval $[0,1]$  
such that
for any element $\tau \in \Gamma_A$,
there exists an
$A$-adic PL function $f_\tau$ and a
finite set $S_\tau \subset X_A$ satisfying the following properties: 
\begin{enumerate}
\renewcommand{\theenumi}{\roman{enumi}}
\renewcommand{\labelenumi}{\textup{(\theenumi)}}
\item
$f_\tau(\rho_A(x)) = \rho_A(\tau(x))$ for $x \in X_A\backslash S_\tau$,
\item 
$\frac{d f_\tau}{dt}(\rho_A(x)) = D_\tau(x)$
for $x \in X_A\backslash S_\tau$.
\end{enumerate}
\end{theorem}

In \cite{Brown}, K. S. Brown has extended the groups $V_N, 1<N \in \N$ 
to triplets $F_N \subset T_N \subset V_N$ of infinite discrete groups.
In the final section, 
we will  generalize the triplet 
to the triplet  
$F_A \subset T_A \subset \Gamma_A$
of infinite discrete groups.

Throughout the paper, 
we denote by $\N$ and by $\Zp$ 
the set of positive integers
and 
the set of nonnegative integers, respectively.

\section{Preliminaries}
Let $A=[A(i,j)]_{i,j=1}^N$ 
be an $N\times N$ matrix with entries in $\{0,1\}$,
where $1< N \in {\Bbb N}$.
Then $A$ is said to be irreducible if for every pair $(i,j), i,j=1,\dots,N$,
there exists $k\in \N$ such that $A^k(i,j) \ge 1$. 
If $A^m =\id$ for some $m\in \N$,
then $A$ is called a permutation matrix.
Throughout the paper, 
we assume that   
$A$ is irreducible and not any permutations. 
We denote by 
$X_A$ the shift space 
\begin{equation*}
X_A = \{ (x_n )_{n \in {\mathbb{N}}} \in \{1,\dots,N \}^{\mathbb{N}}
\mid
A(x_n,x_{n+1}) =1 \text{ for all } n \in {\mathbb{N}}
\}
\end{equation*}
of the right one-sided topological Markov shift for $A$.
It is a compact Hausdorff space in natural  product topology.
The shift transformation $\sigma_A$ on $X_A$ defined by 
$\sigma_{A}((x_n)_{n \in {\mathbb{N}}})=(x_{n+1} )_{n \in {\mathbb{N}}}$
is a continuous surjection  on $X_A$.
The topological dynamical system 
$(X_A, \sigma_A)$ is called the (right one-sided) topological Markov shift for $A$.
Since $A$ is assumed to be irreducible and not any permutations,
the shift space $X_A$ is homeomorphic to a Cantor discontinuum.

A word $\mu = (\mu_1, \dots, \mu_m)$ for $\mu_i \in \{1,\dots,N\}$
is said to be admissible for $X_A$ 
if $\mu$ appears somewhere in some element $x$ in $X_A$.
The length of $\mu$ is $m$ and denoted by $|\mu|$.
 We denote by 
$B_m(X_A)$ the set of all admissible words of length $m$.
For $m=0$ we denote by $B_0(X_A)$ the empty word $\emptyset$.
We put 
$B_*(X_A) = \cup_{m=0}^\infty B_m(X_A)$ 
the set of admissible words of $X_A$.
For two words 
$\mu = (\mu_1, \dots, \mu_m) \in B_m(X_A),
 \nu = (\nu_1, \dots, \nu_n) \in B_n(X_A)$,
 we denote by 
 $\mu \nu$ 
 the word
 $(\mu_1, \dots, \mu_m,\nu_1, \dots, \nu_n)$.
For a word $\mu = (\mu_1, \dots, \mu_m) \in B_m(X_A)$,
the cylinder set
$U_\mu \subset X_A$  
is defined by
\begin{equation*}
U_\mu = \{ (x_n)_{n \in {\mathbb{N}}} \in X_A 
\mid x_1 = \mu_1, \dots, x_m = \mu_m \}.
\end{equation*}
We put  
\begin{align*}
\Gamma^+_k(\mu) & = 
\{  (\eta_1, \dots,\eta_k) \in B_k(X_A)
\mid 
(\mu_1,\dots,\mu_m, \eta_1, \dots,\eta_k ) \in B_{m+k}(X_A) \}, 
\quad k \in \Zp, \\
\Gamma^+_\infty(\mu) & = 
\{  (x_n)_{n \in \N} \in X_A
\mid 
(\mu_1,\dots,\mu_m, x_1, x_2, \dots ) \in X_A \}
\end{align*}
and
$\Gamma^+_*(\mu) = \cup_{k=1}^\infty \Gamma_k^+(\mu)$
which is called the follower set of $\mu$.
For two words $\mu,\nu \in B_*(X_A)$,
we see that 
$\Gamma^+_*(\mu) =\Gamma^+_*(\nu)$
if and only if
$\Gamma^+_\infty(\mu) =\Gamma^+_\infty(\nu)$.


A homeomorphism $\tau$ on $X_A$ is said to be a {\it cylinder map\/}
if there exist two families
\begin{align*}
\mu(i)
& =(\mu_1(i),\mu_2(i),\dots,\mu_{k_i}(i))\in B_{k_i}(X_A),
\qquad i=1,\dots,m, \\
\nu(i)
& =(\nu_1(i),\nu_2(i),\dots,\nu_{l_i}(i)) \in B_{l_i}(X_A), 
\qquad i=1,\dots,m
\end{align*}
of words such that
\begin{align}
U_{\nu(i)} \cap U_{\nu(j)} & =  U_{\mu(i)} \cap U_{\mu(j)} =
\emptyset, \quad \text{ for } i \ne j, \label{eq:cylinder1}\\
\cup_{i=1}^m U_{\nu(i)}  & =  \cup_{i=1}^m U_{\mu(i)} = X_A,
\label{eq:cylinder2} \\
\Gamma^+_*(\nu(i)) & = \Gamma^+_*(\mu(i)) \quad \text{ for } i =1,\dots,m, \label{eq:cylinder3}
\end{align}
and
\begin{equation}
\tau(\nu_1(i),\nu_2(i),\dots,\nu_{l_i}(i),x_{l_i+1},x_{l_i+2},\dots )
=    (\mu_1(i),\mu_2(i),\dots,\mu_{k_i}(i),x_{l_i+1},x_{l_i+2},\dots )
\label{eq:cylinder4}
\end{equation}
for
$ (x_{l_i+1},x_{l_i+2},\dots ) \in \Gamma^+_\infty(\nu(i))
$
and
$i=1,\dots,m$.
It is easy to see that the set of cylinder maps forms a subgroup
of the group  $\Homeo(X_A)$
of all homeomorphisms on $X_A$.
\begin{definition}\label{defn:tau}
The {\it continuous full group}\ $\Gamma_A$ of $(X_A,\sigma_A)$
is defined as the group of cylinder maps on $X_A$.
\end{definition}
For a cylinder map $\tau\in \Gamma_A$,
define continuous functions
$k_\tau, l_\tau : X_A \rightarrow \Zp$ 
by
\begin{equation}
k_\tau(x) = k_i \text{ for } x \in U_{\mu(i)}, \qquad
l_\tau(x) = l_i \text{ for } x \in U_{\nu(i)},
\end{equation}
so that they satisfy 
\begin{equation}
\sigma_A^{k_\tau(x)}(\tau(x) )=\sigma_A^{l_\tau(x)}(x)
\quad\text{ for all } x \in X_A. \label{eq:tau}
\end{equation}
Conversely a homeomorphism $\tau$ satisfying the equality
 \eqref{eq:tau} for some continuous functions 
$k_\tau, l_\tau : X_A \rightarrow \Zp$ gives rise to a cylinder map
(cf. (\cite{MaPre2012}).


The Cuntz--Krieger algebra $\OA$ for the matrix $A$ has been defined 
in \cite{CK} 
as the universal $C^*$-algebra generated by 
$N$ partial isometries $S_1,\dots, S_N$ subject to the relations:
\begin{equation} 
\sum_{j=1}^N S_j S_j^* = 1, \qquad
S_i^* S_i = \sum_{j=1}^N A(i,j) S_jS_j^*, \quad i=1,\dots,N. \label{eq:CK}
\end{equation} 
The algebra $\OA$ is known to be the unique $C^*$-algebra 
subject to the above relations.
For a word $\mu=(\mu_1,\dots, \mu_k)$ 
with 
$\mu_i \in \{1,\dots,N \}$, 
we denote the product 
$S_{\mu_1} \cdots S_{\mu_k}$ 
by 
$S_\mu$.
Then $S_\mu \ne 0$ if and only if $\mu \in B_*(X_A)$. 
Let
$
C^*(S_\mu S_\mu^*; \mu \in B_*(X_A))
$
 be the $C^*$-subalgebra of $\OA$
generated by the projections 
of the form
$S_\mu S_\mu^*, \mu \in B_*(X_A)$,
which we  denote by $\DA$.
 It is  isomorphic to the commutative $C^*$-algebra 
$C(X_A)$ of all complex valued continuous functions on 
$X_A$ through the correspondence
$ 
S_\mu S_\mu^* \in \DA \longleftrightarrow \chi_\mu \in C(X_A)
$
where
$\chi_\mu$ denotes the characteristic function on $X_A$
for the cylinder set 
$
U_\mu
$
for $\mu \in B_*(X_A)$.  
We will identify $C(X_A)$ with the subalgebra 
$\DA$ of $\OA$.
It is well-known 
that the algebra $\DA$ is maximal abelian in $\OA$
(\cite[Remark 2.18]{CK}).
We denote by 
 $U(\OA)$ and $U(\DA)$
the group of unitaries in $\OA$ and 
the group of unitaries in $\DA$, respectively.  
The normalizer 
$N(\OA, \DA)$ 
of $\DA$ in $\OA$ 
is defined by 
\begin{equation*}
N(\OA, \DA)  = \{ u \in U(\OA) \mid u \DA u^* = \DA \}.
\end{equation*} 

The 
\'{e}tale groupoid $G_A$ 
for the topological Markov shift
$(X_A,\sigma_A)$ 
is given by
\begin{equation*}
G_A =
\{ (x, n,y) \in X_A \times \Zp \times X_A
\mid \text{there exist } k,l \in \Zp; \, n = k-l,\, \sigma_A^k(x) = \sigma_A^l(y) \}.
\end{equation*}
The topology of $G_A$ is generated by the sets 
\begin{equation*}
\{ (x, k-l, y ) \in G_A \mid
x \in V, y \in W, \, \sigma_A^k(x) = \sigma_A^l(y) \}
\end{equation*}
for open sets $V, W \subset X_A$ and $k,l \in \Zp$.
Two elements 
$(x,n,y), (x', n', y') \in G_A$
are composable if and only if $y = x'$
and the product and the inverse are given by
\begin{equation*}
(x,n,y)\cdot (x', n', y') = (x,n +n', y'),
\qquad 
(x,n,y)^{-1} = (y, -n, x).
\end{equation*} 
The unit space 
$G_A^{(0)}$ is defined by $\{(x,0,x) \mid x\in X_A \}$, which is identified with $X_A$.
The range map, source map $r, s:G_A \longrightarrow G^{(0)}$
are defined by $r(x,n,y) = x, s(x,n,y) = y$
respectively.
A subset $U \subset G_A$ is called a 
$G_A$-set if $r|_{U}, s|_{U}$ are injective.    
For an open $G_A$-set $U$, 
denote by 
$\pi_U$ the homeomorphism $r\circ (s|_{U})^{-1}$ from
$s(U)$ to $r(U)$.  
The topological full group
$[[G_A]]$ of $G_A$ is defined by the group of all homeomorphisms
$\pi_U$ for some compact open $G_A$-set $U$
such that 
$s(U) =r(U) = G^{(0)}$  
(see \cite{Matui2015}).
The groupoid $C^*$-algebra $C^*_r(G_A)$ of the groupoid $G_A$
is nothing but the Cuntz--Krieger algebra $\OA$ and 
the commutative $C^*$-algebra $C(G_A^{(0)})$
on the unit space $G_A^{(0)}$ is $\DA$.
The topological full group
$[[G_A]]$ of the 
\'{e}tale groupoid $G_A$ 
for the topological Markov shift
$(X_A,\sigma_A)$
is naturally identified with the continuous full group
$\Gamma_A$ (\cite{Matui2015}).


\begin{lemma}\label{lem:ut}
For $\tau \in \Gamma_A$, 
there exist $u_\tau \in N(\OA,\DA)$
and
$\mu(i), \nu(i) \in B_*(X_A), i=1,\dots,m$  
such that 
 \begin{enumerate}
\renewcommand{\theenumi}{\arabic{enumi}}
\renewcommand{\labelenumi}{\textup{(\theenumi)}}
\item
$ u_\tau = \sum_{i=1}^m S_{\mu(i)}S_{\nu(i)}^*$
and
{\renewcommand{\theenumi}{(\alpha{enumi})}
\renewcommand{\labelenumi}{\textup{(\theenumi)}}
\begin{enumerate}
\item
$S_{\nu(i)}^*S_{\nu(i)}
 = S_{\mu(i)}^*S_{\mu(i)}, \quad i=1,\dots,m,
$
 \item
$\sum_{i=1}^m S_{\nu(i)}S_{\nu(i)}^* 
 =\sum_{i=1}^m S_{\mu(i)}S_{\mu(i)}^*=1.
$
\end{enumerate}}
\item
$f \circ \tau^{-1} = u_\tau f u_\tau^*$ for $f \in \DA$.
\end{enumerate}
\end{lemma}
\begin{proof}
Since $\tau$ is a cylinder map,
 there exist two families of words
 $\mu(1), \dots, \mu(m)$ and
 $\nu(1), \dots, \nu(m)$
 satisfying
 \eqref{eq:cylinder1},
 \eqref{eq:cylinder2},
 \eqref{eq:cylinder3}
 and
 \eqref{eq:cylinder4}.
Hence we have
\begin{equation*}
\sum_{i=1}^m S_{\nu(i)}S_{\nu(i)}^* =
\sum_{i=1}^m S_{\mu(i)}S_{\mu(i)}^*=1,
\qquad 
S_{\nu(i)}^*S_{\nu(i)} = S_{\mu(i)}^*S_{\mu(i)}, \quad i=1,\dots,m.
\end{equation*}
By putting
$ u_\tau = \sum_{i=1}^m S_{\mu(i)}S_{\nu(i)}^*$
we see that $u_\tau$ belongs to
$N(\OA,\DA)$
and
satisfies
$\chi_{U_\eta} \circ\tau^{-1}
= u_\tau \chi_{U_\eta} u_\tau^*$
for all $\eta \in B_*(X_A)$
where
$\chi_{U_\eta}$
is identified with
$
S_\eta S_\eta^*,
$
so that
$f \circ \tau^{-1} = u_\tau f u_\tau^*$ 
for all $f \in \DA$.
\end{proof}
As in \cite[Theorem 1.2]{MaPacific}, \cite[Proposition 5.6]{MatuiPLMS},
there exists a short exact sequence
\begin{equation*}
1 \longrightarrow U(\DA) 
\longrightarrow N(\OA,\DA) 
\longrightarrow \Gamma_A 
\longrightarrow 1
\end{equation*}
that splits.

It has been proved by the second named author \cite{Matui2015} that
the homology group $H_0(G_A)$ of the groupoid $G_A$ is isomorphic
to the $K_0$-group 
$K_0(\OA) = {\mathbb{Z}}^N / (I - A^t){\mathbb{Z}}^N$ 
of the $C^*$-algebra $\OA$.
He has proved that 
the group $\Gamma_A$ is simple if 
and only if $H_0(G_A)$ is $2$-divisible.
He has also proved that $\Gamma_A$ is finitely presented and
its commutator subgroup
$D(\Gamma_A)$ is always simple.
As the group $\Gamma_A$ is non-amenable (\cite{MaDCDS}, \cite{Matui2015}),
we see
\begin{theorem}[\cite{Matui2015}]
The group $\Gamma_A$ is a countably infinite,
non-amenable,  finitely presented discrete group.
It is simple if and only if 
the group ${\mathbb{Z}}^N / (I - A^t){\mathbb{Z}}^N$ 
 is $2$-divisible.
\end{theorem}
It has been shown that 
for two irreducible square matrices $A$ and $B$,
the groups $\Gamma_A$ and $\Gamma_B$ are isomorphic
if and only if  
the $C^*$-algebras $\OA$ and $\OB$ are isomorphic
and $\det(1-A) = \det(1-B)$ 
(\cite{MatsumotoMatui2014}).
Hence the family $\{ \Gamma_A\}$ of our groups supply us many mutually non-isomorphic countably infinite, non-amenable, finitely presented simple groups.

\section{Realization of $\OA$ on $L^2([0,1])$}
The Higman--Thompson group 
$V_N, 1 < N \in {\mathbb{N}}$ 
is represented as the group of right continuous piecewise linear bijective functions 
$f: [0,1) \longrightarrow [0,1)$
having finitely many singularities such that
all singularities of $f$ are in ${\mathbb{Z}}[\frac{1}{N}]$,
the derivative of $f$ at any non-singular point is
$N^k$ for some $k \in {\mathbb{Z}}$ and 
$f({\mathbb{Z}}[\frac{1}{N}] \cap [0,1)) 
= {\mathbb{Z}}[\frac{1}{N}] \cap [0,1)$.
In order to represent our group
$\Gamma_A$ as a group of piecewise linear functions on $[0,1)$,
we will represent the algebra $\OA$ on 
  the Hilbert space $H$
of the square integrable functions $L^2([0,1])$ 
on $[0,1]$ with respect to the Lebesgue measure
in the following way.
We note that the essentially bounded measurable functions $L^\infty([0,1])$
 act on $H$ by left multiplication.

Since $A$ is irreducible and not any permutations,
 its  Perron--Frobenius eigenvalue written  $\beta $ 
 is greater than one.
By Ruelle's Perron-Frobenius theory for Markov chains,
there uniquely exists a faithful Borel probability measure 
$\phi$ on $X_A$ satisfying the equality
\begin{equation}
\int_{x \in X_A} g(x) d\phi(\sigma_A(x)) = \beta \int_{x \in X_A} g(x) d\phi(x),
\qquad g \in C(X_A) \quad (\text{see } \cite{PP}). \label{eq:RPF} 
\end{equation}
Under the identification between $C(X_A)$ and the $C^*$-subalgebra $\DA$ of $\OA$,
the probability measure $\phi$ on $X_A$ is regarded as a continuous linear functional on $\DA$,
which is still denoted by $\phi$. 
Let  $\lambda_A:\DA \rightarrow \DA$ be the positive operator 
defined by 
$\lambda_A(g) = \sum_{i=1}^N S_i^* g S_i$ for $g \in \DA$.
Since the characteristic function
$\chi_\mu$ on $X_A$ for the cylinder set of an admissible word $\mu \in B_*(X_A)$  
is regarded as the projection $S_\mu S_\mu^* $ in $\DA$,
the identity \eqref{eq:RPF} implies  
\begin{equation}
\varphi(\lambda_A(g)) = \beta \varphi(g), \qquad g \in \DA \label{eq:lambdaA}
\end{equation}
so that  the equality
\begin{equation}
\sum_{j=1}^N
A(i,j) \varphi(S_j S_j^*)
=\beta\varphi(S_i S_i^*), \qquad
 i=1,\dots,N \label{eq:eigen}
\end{equation}
holds. 
Put
$
p_j = \varphi(S_jS_j^*), \, j=1,\dots,N.
$
The equality \eqref{eq:eigen} 
means that the  vector
\begin{math}
\begin{bmatrix}
p_1\\
\vdots\\
p_N
\end{bmatrix}
\end{math}
is a unique normalized positive eigenvector 
for  the Perron--Frobenius eigenvalue $\beta$.
For $i,j=1,2,\dots,N$,
put
$
p_{ij} = \varphi(S_i S_j S_j^* S_i^*)
$
so that
\begin{equation*}
p_{ij} = \frac{1}{\beta^2}\varphi(S_j^*S_i^* S_i S_j)
       =\frac{1}{\beta^2}A(i,j)\varphi(S_j^* S_j)
       =\frac{1}{\beta}A(i,j)p_j.
\end{equation*}
We set for $i,j = 1,2,\dots,N$,
\begin{equation*}
p(0) = 0, \qquad p(i)  = \sum_{k=1}^i p_k, \qquad 
q(0,0) = q(i,0) =0, \qquad q(i,j)  = \sum_{k=1}^j p_{ik}
\end{equation*}
and define the intervals $I_i, \, I_{ij}$ in $[0,1)$ by
\begin{align}
I_i & = [p(i-1), p(i)),  \label{eq:Ii}\\
I_{ij} & = [p(i-1)+q(i,j-1), p(i-1)+q(i,j)).\label{eq:Iij}
\end{align}
The latter interval $I_{ij}$ is empty if $A(i,j) =0$.
We set 
\begin{align*}
l(I_i) & = p(i-1), \qquad r(I_i) = p(i),\\
l(I_{ij}) & = p(i-1)+q(i,j-1), \qquad r(I_{ij}) = p(i-1)+q(i,j)
\end{align*}
so that 
\begin{equation*}
I_i = [l(I_i), r(I_i)), \qquad
I_{ij} =[l(I_{ij}),r(I_{ij})).
\end{equation*}
\begin{lemma} Keep the above notations.
 \begin{enumerate}
\renewcommand{\theenumi}{\roman{enumi}}
\renewcommand{\labelenumi}{\textup{(\theenumi)}}
\item  $[0,1) = \sqcup_{i=1}^N I_i$ : disjoint union.
\item  $I_i = \sqcup_{j=1}^N I_{ij}$ : disjoint union.
\end{enumerate}
\end{lemma}
\begin{proof}
(i) is clear.
(ii) 
Let 
$N_i =\Max\{ j=1,\dots,N \mid A(i,j) =1 \}.$ 
As we have
\begin{equation*}
q(i,N_i) = \sum_{k=1}^{N_i} p_{ik} 
       = \frac{1}{\beta}\sum_{k=1}^{N_i} A(i,k)p_{k} = p_i,  
\end{equation*}
the equality
$p(i-1) + q(i,N_i) = p(i)$ holds so that
$r(I_{i,N_i}) = r(I_i)$.
As the intervals 
$I_{ij}, I_{ij'}$ are disjoint for $j \ne j'$,
one easily sees that 
$I_i = \sqcup_{j=1}^{N_i} I_{ij}= \sqcup_{j=1}^N I_{ij}$.
\end{proof}
We define right continuous functions $f_A, g_1,\dots,g_N$
in the following way.
The function $f_A:[0,1) \longrightarrow [0,1)$ is defined by
\begin{equation*}
f_A(x) =  \beta(x - l(I_{ij})) + l(I_j) \quad \text{ for } x \in I_{ij}
\end{equation*}  
so that $f_A$ is linear on $I_{ij}$ with slope $\beta$ and
$f_A(I_{ij}) = I_j$.
We set 
\begin{equation*}
J_i = 
\bigcup_{\substack{j=1,\dots,N\\
            A(i,j) =1}}
I_j.
\end{equation*}
The function $g_i:J_i \longrightarrow I_i$ for each 
$i=1,\dots,N$ 
is defined by
\begin{equation*}
g_i(x) =  \frac{1}{\beta}(x - l(I_{j})) + l(I_{ij}) 
           \quad \text{ for } x \in I_{j} \text{ with } A(i,j) =1
\end{equation*}  
so that $g_i$ is linear on $I_{j}$ 
for $A(i,j) =1$  with slope $\frac{1}{\beta}$ 
and
$g_i(I_{j}) = I_{ij},\,
g_i(J_i) = I_i$.
The following lemma is direct.
\begin{lemma} For $i=1,\dots,N$, we have
 \begin{enumerate}
\renewcommand{\theenumi}{\roman{enumi}}
\renewcommand{\labelenumi}{\textup{(\theenumi)}}
\item
$f_A (g_i (x)) = x $ for $x \in J_i$.
\item
$g_i(f_A(x)) = x $  for $x \in I_i$.
\end{enumerate}
\end{lemma}
For a measurable subset $E$ of $[0,1)$,
denote by 
$\chi_E$ the multiplication operator on $H$ 
of the characteristic function of $E$. 
Define  the bounded linear operators
$T_{f_A}$, $T_{g_i},i= 1,\dots,N$
on $H$ by
\begin{equation*}
(T_{f_A} \xi)(x) 
  = \xi(f_A(x)), 
  \qquad
(T_{g_i} \xi)(x) 
 = \chi_{J_i}(x)\xi(g_i(x))  \qquad \text{for }
 \xi \in H, x \in [0,1). 
\end{equation*}
The following lemma is straightforward:
\begin{lemma} Keep the above notations. 
We have
\begin{enumerate}
\renewcommand{\theenumi}{\roman{enumi}}
\renewcommand{\labelenumi}{\textup{(\theenumi)}}
\item 
$T_{f_A}^* = \frac{1}{\beta}\sum_{i=1}^{N} T_{g_i}.$
\item 
$T_{f_A}^* T_{f_A} 
= \frac{1}{\beta} \sum_{i=1}^N \chi_{J_i}.$
\item 
$T_{g_i}^* T_{g_i} = \beta \chi_{I_i}$ 
for $i=1,\dots,N$
and hence $\sum_{i=1}^N T_{g_i}^* T_{g_i} = \beta 1$.
\item 
$T_{g_i} T_{g_i}^* = 
\beta \chi_{J_i}
$ 
for
$i=1,\dots,N.
$
\end{enumerate}
\end{lemma}
We define
the operators $s_i, i=1,\dots,N$ on $H$ by setting
 \begin{equation*}
 s_i = \frac{1}{\sqrt{\beta}} T_{g_i}^*, \qquad i=1,\dots,N.
 \end{equation*}
 By the above lemma, we have
\begin{proposition}\label{prop:CK}
The operators
$s_i, i=1,\dots,N$
are  partial isometries such that
\begin{equation*}
s_i s_i^*  = \chi_{I_i}, \qquad
s_i^* s_i  = \chi_{J_i}, \qquad i=1,\dots,N.
\end{equation*}      
Hence they satisfy the relations
\begin{equation*}
\sum_{j=1}^N s_j s_j^*  = 1, \qquad
s_i^* s_i  = \sum_{j=1}^N A(i,j) s_j s_j^*, \qquad i=1,\dots,N.
\end{equation*}      
Therefore the correspondence
$S_i \longrightarrow s_i, i=1,\dots,N$
gives rise to an isomorphism from 
the Cuntz--Krieger algebra $\OA$
 onto 
the $C^*$-algebra
$C^*(s_1,\dots,s_N)$ on $H$.
\end{proposition}

\section{$A$-adic PL functions}
By Proposition \ref{prop:CK},
we may represent $\OA$ on $H$ by identifying 
$S_i$ with $s_i$ for $i=1,\dots,N$.
In this section,
we will define PL (piecewise linear) functions
on $[0,1)$ associated to the topological Markov shift 
$(X_A,\sigma_A)$.
For $\mu = (\mu_1,\dots,\mu_n) \in B_n(X_A)$,
define 
\begin{equation*}
l(\mu)  = 
\sum_{\substack{
\nu \in B_n(X_A)\\
\nu \prec \mu
}}
\varphi(S_\nu S_\nu^*),
\qquad
r(\mu) = l(\mu) + \varphi(S_\mu S_\mu^*).
\end{equation*}
Put the interval
\begin{equation*}
I_\mu  
= [l(\mu), r(\mu)).
\end{equation*}
The following lemma is clear.
\begin{lemma}\label{lem:Imu}
For each $n \in {\mathbb{N}}$ we have
 \begin{enumerate}
\renewcommand{\theenumi}{\roman{enumi}}
\renewcommand{\labelenumi}{\textup{(\theenumi)}}
\item $I_\mu \cap I_\nu = \emptyset$ for $\mu, \nu \in B_n(X_A)$
with $\mu \ne \nu$. 
\item $\cup_{\mu\in B_n(X_A)} I_\mu = [0,1)$.
\end{enumerate}
\end{lemma}
For $\mu =(\mu_1,\dots,\mu_n) \in B_n(X_A)$, 
we note that the following equalites hold
\begin{equation}
\varphi(S_\mu S_\mu^*) 
=  \frac{1}{\beta^n}\varphi(S_\mu^* S_\mu) 
=  \frac{1}{\beta^n}\varphi(S_{\mu_n}^* S_{\mu_n}) 
=  \frac{1}{\beta^n}\sum_{j=1}^N A(\mu_n,j)p_j. \label{eq:varphmu}
\end{equation}
For $i,j = 1,\dots,N$ with
$A(i,j) =1$,
we apply 
\eqref{eq:varphmu}
for $\mu = i,\,  (i,j)$
so that  
\begin{align*}
l(i)  
& = {\sum_{j<i}}
\varphi(S_j S_j^*)
=\sum_{j=1}^{i-1} p_j 
= p(i-1),\\
r(i)
& = l(i) + \varphi(S_i S_i^*)
  = p(i-1) + p_i = p(i)
\end{align*}
and
\begin{align*}
l(i,j)  
& = 
{\sum_{(\mu_1, \mu_2)\prec (i,j)}}
\varphi(S_{\mu_1}S_{\mu_2} S_{\mu_2}^*S_{\mu_1}^*)
=
{\sum_{(\mu_1,\mu_2)\prec (i,j)}}
p_{\mu_1\mu_2} \\
& =\sum_{\mu_1=1}^{i-1}\sum_{\mu_2=1}^N p_{\mu_1\mu_2} 
 + \sum_{\mu_2=1}^{j-1} p_{i\mu_2} \\
& =\sum_{\mu_1=1}^{i-1}\sum_{\mu_2=1}^N 
   A(\mu_1,\mu_2) \frac{1}{\beta}p_{\mu_2} 
   + q(i,j-1) =p(i-1) + q(i,j-1),\\ 
r(i,j)
& = l(i,j) + \varphi(S_i S_j S_j^* S_i^*)
  =p(i-1) + q(i,j-1) +p_{ij}
  =p(i-1) + q(i,j).
\end{align*}
Hence we see that
\begin{align*}
[l(i),r(i)) & = [p(i-1),p(i)) = I_i : \text{ the interval defined in } 
\eqref{eq:Ii},\\
[l(i,j),r(i,j)) & = [p(i-1) +q(i,j-1),p(i-1)+q(i,j)) = I_{ij} 
: \text{ the interval defined in } 
\eqref{eq:Iij}.
\end{align*}
\begin{lemma}
For $\mu = (\mu_1,\dots,\mu_m) \in B_m(X_A)$,
we have
\begin{equation*}
f_A(I_\mu) = I_{\mu_2\cdots\mu_m}
\quad
\text{ and hence }
\quad
f_A^{m-1}(I_\mu) = I_{\mu_m}(=[l(\mu_m),r(\mu_m))).
\end{equation*}
\end{lemma}
\begin{proof}
The algebra $\OA$ is represented on $H$ by identifying 
$S_i$ with $s_i$ for $i=1,\dots,N$.
We then see
\begin{equation*}
S_\mu S_\mu^* = \chi_{I_\mu}
\quad
\text{ and  }
\quad
\lambda_A(S_\mu S_\mu^*) = \chi_{f_A(I_\mu)}.
\end{equation*}
Since
$S_{\mu_1}^*S_{\mu_1}\ge S_{\mu_2}S_{\mu_2}^*$,
we have 
\begin{align*}
\lambda_A(S_\mu S_\mu^*) 
& = S_{\mu_1}^*S_{\mu_1}S_{\mu_2}\cdots S_{\mu_m}S_{\mu_m}^*
\cdots S_{\mu_2}^*S_{\mu_1}^*S_{\mu_1}\\
& = S_{\mu_2}\cdots S_{\mu_m}S_{\mu_m}^*
\cdots S_{\mu_2}^*
\end{align*}
so that 
$
\chi_{I_{\mu_2\cdots\mu_m}}
=
\chi_{f_A(I_\mu)}.
$
\end{proof}
\begin{lemma}
For 
$\mu = (\mu_1,\dots,\mu_m) \in B_m(X_A)$,
$\nu = (\nu_1,\dots,\nu_n) \in B_n(X_A)$,
the condition
$S_\mu^* S_\mu = S_\nu^*S_\nu$ 
implies
\begin{equation}
\frac{r(\mu) -l(\mu)}{r(\nu) -l(\nu)} = 
\beta^{n -m}. \label{eq:slope}
\end{equation}
\end{lemma}
\begin{proof}
Since
$
r(\mu) -l(\mu)
=\varphi( S_\mu S_\mu^*)
=\frac{1}{\beta^m} \varphi(S_\mu^*S_\mu)
$
and similarly
$
r(\nu) -l(\nu)
=\frac{1}{\beta^n} \varphi(S_\nu^*S_\nu),
$
the condition
$S_\mu^* S_\mu = S_\nu^*S_\nu$ 
implies
\eqref{eq:slope}.
\end{proof}

\begin{lemma}\label{lem:five}
For 
$\mu = (\mu_1,\dots,\mu_m) \in B_m(X_A)$,
$\nu = (\nu_1,\dots,\nu_n) \in B_n(X_A)$,
the following five conditions are equivalent:
 \begin{enumerate}
\renewcommand{\theenumi}{\roman{enumi}}
\renewcommand{\labelenumi}{\textup{(\theenumi)}}
\item
$\Gamma^+_*(\mu) = \Gamma^+_*(\nu)$.
\item
$S_\mu^* S_\mu = S_\nu^*S_\nu$. 
\item
$S_{\mu_m}^* S_{\mu_m} = S_{\nu_n}^*S_{\nu_n}$. 
\item
$f_A^m(I_\mu)= f_A^n(I_\nu)$.
\item
$f_A(I_{\mu_m})= f_A(I_{\nu_n})$.
\end{enumerate}
\end{lemma}
\begin{proof}
For 
$\mu = (\mu_1,\dots,\mu_m) \in B_m(X_A)$,
the identites
\begin{equation*}
\chi_{f_A^m(I_\mu)} 
=
\chi_{f_A(I_{\mu_m})} 
=
\lambda_A(S_{\mu_m}S_{\mu_m}^*)
=
S_{\mu_m}^*S_{\mu_m}
=
S_\mu^* S_\mu 
\end{equation*}
hold.
They imply the desired assertion.
\end{proof}
\begin{definition}\label{defn:rectangle}
\begin{enumerate}
\renewcommand{\theenumi}{\roman{enumi}}
\renewcommand{\labelenumi}{\textup{(\theenumi)}}
\item 
For a word $\nu \in B_*(X_A)$,
an interval $[x_1,x_2)$ in $[0,1)$ is said to be 
an $A$-{\it adic interval} for $\nu$ 
if 
$x_1 =l(\nu) $ and
$x_2 = r(\nu)$.
\item
A rectangle
$I \times J$ in 
$[0,1)\times [0,1)$
 is said to be 
an $A$-{\it adic rectangle}\
if  both the intervals  
$I, J$ 
are $A$-adic intervals for 
some words
$\nu \in B_n(X_A), \mu \in B_m(X_A)$,
respectively such that
\begin{equation*}
I = [l(\nu), r(\nu)), \qquad
J = [l(\mu), r(\mu))
\qquad
\text{ and }
\qquad
f_A^n(I) = f_A^m(J).
\end{equation*}
\item
For two partitions
\begin{align*}
0=& x_0 < x_1 < \dots <x_{m-1} < x_m = 1,\\
0=& y_0 < y_1 < \dots <y_{m-1} < y_m = 1
\end{align*}
of $[0,1)$,
put
$$
I_p = [x_{p-1},x_p),\qquad
J_p = [y_{p-1},y_p) \quad
\text{ for }
\quad
p=1,2,\dots,m.
$$
The partition
$I_p \times J_q, p,q =1,\dots,m$ of 
$[0,1)\times [0,1)$
is said to be 
an $A$-{\it adic pattern of rectangles}\ 
if there exists a permutation 
$\sigma $ on $\{1,2,\dots,m\}$
such that
the rectangles
$I_p \times J_{\sigma(p)}$
are $A$-adic rectangles for all $p=1,2,\dots,m$.
\end{enumerate}
\end{definition}
For an 
$A$-adic pattern of rectangles above,
the slopes of its diagonals 
\begin{equation*}
s_p = \frac{y_{\sigma(p)} -y_{\sigma(p)-1}}{x_p -x_{p-1}}, 
\qquad p=1,2,\dots,m
\end{equation*}
are said to be {\it rectangle slopes}.

\begin{definition}\label{defn:PL}
A piecewise linear function $f$ on $[0,1)$
is called an $A$-{\it adic PL function}\ 
if $f$ is a right continuous bijection on $[0,1)$
such that there exists an $A$-adic pattern of rectangles 
$I_p \times J_p, p=1,2,\dots,m$
where
$I_p = [x_{p-1}, x_p), J_p = [y_{p-1},y_p), p=1,\dots,m$ 
with a permutation $\sigma$ on $\{1,2,\dots,m \}$
such that
\begin{equation*}
f(x_{p-1}) = y_{\sigma(p)-1},
\qquad
f_{-}(x_p) = y_{\sigma(p-1)+1},
\qquad
 p=1,2,\dots,m
\end{equation*}
where
$
f_{-}(x_p) = \lim_{h \to 0+} f(x_p -h),
$
and $f$ is linear on $[x_{p-1}, x_p)$
with slope
$\frac{y_{\sigma(p)} -y_{\sigma(p)-1}}{x_p -x_{p-1}}$
for
$ p=1,2,\dots,m$.
\end{definition}

\begin{lemma}
The composition of two $A$-adic PL functions
and the inverse function of an $A$-adic PL function
are also  $A$-adic PL functions.
\end{lemma}
By the above lemma, 
the set of $A$-adic PL functions forms a group
under compositions of functions.

\begin{definition}\label{defn:PLgroup}
We denote by $\Gamma_A^{\PL}$
the group of $A$-adic PL functions.
\end{definition}

The following proposition is immediate 
by definition of $A$-adic PL functions.
\begin{proposition}\label{prop:rectanglePL}
An $A$-adic PL function naturally 
gives rise to
an $A$-adic pattern of rectangles,
whose rectangle slopes are the slopes of the $A$-adic PL function.
Conversely, 
an $A$-adic pattern of rectangles
gives rise to
an $A$-adic PL function by taking its diagonal lines  of the 
rectangles.
\end{proposition}

\section{$A$-adic Tables}
For two words 
$\mu=(\mu_1,\dots,\mu_m) \in B_m(X_A), 
\nu=(\nu_1,\dots,\nu_n) \in B_n(X_A)
$
with
$U_{\mu} \cap U_{\nu} =\emptyset$,
we write 
$
\mu \prec \nu 
$
if
$
\mu_1 = \nu_1, \dots,\mu_k = \nu_k
$
and
$
\mu_{k+1}<\nu_{k+1}
$
for some
$k.$
 Nekrashevych in \cite{Nek} 
has introduced a notion of table to represent elements of the Higman--Thompson group
$V_N$.
We will generalize the Nekrashevych's notion to  
a notion of $A$-adic table in order to represent elements of 
the continuous full group $\Gamma_A$.
\begin{definition}\label{defn:table}
An $A$-{\it adic table} is a matrix $T$
\begin{equation*}
T=
\begin{bmatrix}
\mu(1) & \mu(2) & \cdots & \mu(m) \\
\nu(1) & \nu(2) & \cdots & \nu(m) 
\end{bmatrix}
\end{equation*}
for 
$\mu(i), \nu(i) \in B_*(X_A), i=1,\dots, m$
such that
\begin{enumerate}
\renewcommand{\labelenumi}{(\alph{enumi})}
\item $\Gamma^+_*(\nu(i)) = \Gamma^+_*(\mu(i)), i=1,\dots,m$,
\item $X_A = \sqcup_{i=1}^m U_{\nu(i)}
           = \sqcup_{i=1}^m U_{\mu(i)} :$
               disjoint unions.
\end{enumerate}
\end{definition}
Since
the words
$\nu(i), i=1,\dots,m$
satisfy
$U_{\nu(i)}\cap U_{\nu(j)} =\emptyset$
for $i \ne j$,
we may reorder them such as  
$\nu(1) \prec \nu(2) \prec \cdots \prec  \nu(m)$.
As the above two conditions (a), (b) 
are equivalent to the conditions (a), (b) in
Lemma \ref{lem:ut} (1) respectively, 
we have
\begin{lemma}
For an element $\tau \in \Gamma_A$,
let words $\mu(i), \nu(i),i=1,\dots,m$ 
and
the unitary
$ u_\tau = \sum_{i=1}^m S_{\mu(i)}S_{\nu(i)}^*$
satisfy the conditions (1) and (2) in Lemma \ref{lem:ut}.
Then the matrix
\begin{equation*}
T =
\begin{bmatrix}
\mu(1) & \mu(2) & \cdots & \mu(m) \\
\nu(1) & \nu(2) & \cdots & \nu(m) 
\end{bmatrix}
\end{equation*}
is an $A$-adic table. 
\end{lemma}
The $A$-adic table $T$ above is called a representation of $\tau$.
It is also called that $T$ represents $\tau$.

For an $A$-adic table
\begin{math}
T=
\bigl[
\begin{smallmatrix}
\mu(1) & \mu(2) & \cdots & \mu(m) \\
\nu(1) & \nu(2) & \cdots & \nu(m) 
\end{smallmatrix}
\bigr]
\end{math}
and $i=1,2,\dots,m$,
let 
$\eta(i,j)\in B_*(X_A), j=1,\dots, n_i$
be a family of (possibly empty) words satisfying 
the following three conditions:
\begin{enumerate}
\renewcommand{\theenumi}{\roman{enumi}}
\renewcommand{\labelenumi}{\textup{(\theenumi)}}
\item
$\eta(i,1) \prec \eta(i,2) \prec \cdots \prec \eta(i,n_i),$
\item
$\eta(i,j) \in \Gamma_*^+(\nu(i))$ for $j=1,\dots,n_i$,
\item
$U_{\nu(i)} = \cup_{j=1}^{n_i}U_{\nu(i)\eta(i,j)}.$
\end{enumerate}
Since
$\Gamma^+_*(\nu(i))=\Gamma^+_*(\mu(i))$,
one has 
$\eta(i,j) \in \Gamma_*^+(\mu(i))$
and
$U_{\mu(i)} = \cup_{j=1}^{n_i}U_{\mu(i)\eta(i,j)}$.
Put
\begin{equation}
\nu(i,j) = \nu(i)\eta(i,j),\quad
\mu(i,j) = \mu(i)\eta(i,j), \quad j=1,\dots, n_i, \, i=1,\dots,m.
\label{eq:nuij}
\end{equation}
Then the $2 \times m$ matrix
\begin{equation*}
\begin{bmatrix}
\mu(1,1) & \cdots &\mu(1,n_1) & 
\mu(2,1) & \cdots &\mu(2,n_2) & \cdots & 
\mu(m,1) & \cdots &\mu(m,n_m) \\
\nu(1,1) & \cdots &\nu(1,n_1) & 
\nu(2,1) & \cdots &\nu(2,n_2) & \cdots & 
\nu(m,1) & \cdots &\nu(m,n_m)  
\end{bmatrix}
\end{equation*}
is an $A$-adic table, which is called an {\it expansion} of $T$.
Let us denote by 
$\approx$ the equivalence relation in the $A$-adic tables generated by the expansions.
This means that two $A$-adic tables 
\begin{equation*}
T =\begin{bmatrix}
\mu(1) & \mu(2) & \cdots & \mu(m) \\
\nu(1) & \nu(2) & \cdots & \nu(m) 
\end{bmatrix},
\qquad
T' =\begin{bmatrix}
\mu'(1) & \mu'(2) & \cdots & \mu'(m') \\
\nu'(1) & \nu'(2) & \cdots & \nu'(m') 
\end{bmatrix},
\end{equation*}
are equivalent and written 
$T \approx T'$ if
there exists a finite sequence 
$T_1,T_2,\dots,T_k$ of $A$-adic tables such that 
$T=T_1, T' = T_k$
and
$T_i$ is an expansion of $T_{i+1}$, or
$T_{i+1}$ is an expansion of $T_{i}$.

\begin{lemma}\label{lem:tauT}
For $\tau, \tau' \in \Gamma_A$,
let $T,\, T'$ be $A$-adic tables 
representing $\tau, \, \tau'$ respectively.
Then  $\tau= \tau'$ if and only if $T \approx T'$.
\end{lemma}
\begin{proof}
Let 
$T, \, T'$ be the matrices 
\begin{equation*}
T =\begin{bmatrix}
\mu(1) & \mu(2) & \cdots & \mu(m) \\
\nu(1) & \nu(2) & \cdots & \nu(m) 
\end{bmatrix},
\qquad
T' =\begin{bmatrix}
\mu'(1) & \mu'(2) & \cdots & \mu'(m') \\
\nu'(1) & \nu'(2) & \cdots & \nu'(m') 
\end{bmatrix}.
\end{equation*}
Suppose that $T'$ is an expansion of $T$.
We write $T'$ as
\begin{equation*}
\begin{bmatrix}
\mu(1,1) & \cdots &\mu(1,n_1) & 
\mu(2,1) & \cdots &\mu(2,n_2) & \cdots & 
\mu(m,1) & \cdots &\mu(m,n_m) \\
\nu(1,1) & \cdots &\nu(1,n_1) & 
\nu(2,1) & \cdots &\nu(2,n_2) & \cdots & 
\nu(m,1) & \cdots &\nu(m,n_m)  
\end{bmatrix}
\end{equation*}
where
$\mu(i,j)$ and
$\nu(i,j)$ are words for $\eta(i,j)$
as in \eqref{eq:nuij}.
The homeomorphisms
 $\tau$ and $\tau'$ on $X_A$ are induced by the unitaries
$u_T$ and $u_{T'}$ defined by
\begin{equation*}
u_T = \sum_{i=1}^m S_{\mu(i)}S_{\nu(i)}^*
\quad
\text{ and }
\quad
u_{T'} = \sum_{i=1}^{m'} S_{\mu'(i)}S_{\nu'(i)}^*
\end{equation*} 
such as
$
f \circ \tau^{-1}  = \Ad(u_\tau)(f)
$
and
$
f \circ {\tau'}^{-1}  = \Ad(u_{\tau'})(f)
$
for
$
f \in C(X_A)=\DA.
$ 
As
\begin{equation*}
S_{\mu(i)}S_{\nu(i)}^*
=\sum_{j=1}^{n_i}
S_{\mu(i)}S_{\eta(i,j)}S_{\eta(i,j)}^*S_{\nu(i)}^*
=\sum_{j=1}^{n_i}
S_{\mu(i,j)}S_{\nu(i,j)}^*,
\end{equation*}
we have
\begin{equation*}
u_\tau
=\sum_{i=1}^m S_{\mu(i)}S_{\nu(i)}^*
=\sum_{i=1}^m 
S_{\mu(i,j)}S_{\nu(i,j)}^*
= u_{\tau'}
\end{equation*}
so that $\tau = \tau'$.

Conversely, suppose that
$\tau = \tau'$.
Let
$$
K' = \Max\{|\nu'(i)| \mid {1\le i \le m'} \},\qquad
L' = \Max\{|\mu'(i)| \mid {1\le i \le m'} \}.
$$
There exist
admissible words
$\eta(i,j) \in B_*(X_A), j=1,\dots,n_i, \, i=1.\dots,m$
such that
\begin{enumerate}
\renewcommand{\labelenumi}{(\alph{enumi})}
\item
$\eta(i,1) \prec \eta(i,2) \prec \cdots \prec \eta(i,n_i),$
\item
$\eta(i,j) \in \Gamma_*^+(\nu(i))$,
\item
$|\nu(i) \eta(i,j)| \ge K', \, |\mu(i) \eta(i,j)| \ge L',$
\item
$U_{\nu(i)} = \sqcup_{j=1}^{n_i}U_{\nu(i)\eta(i,j)},
$
$U_{\mu(i)} = \sqcup_{j=1}^{n_i}U_{\mu(i)\eta(i,j)}.
$
\end{enumerate}
Put
\begin{equation*}
\nu(i,j) = \nu(i)\eta(i,j),\quad
\mu(i,j) = \mu(i)\eta(i,j), \quad j=1,\dots, n_i, \, i=1,\dots,m
\end{equation*}
and
\begin{equation*}
T^\eta =
\begin{bmatrix}
\mu(1,1) & \cdots &\mu(1,n_1) & 
\mu(2,1) & \cdots &\mu(2,n_2) & \cdots & 
\mu(m,1) & \cdots &\mu(m,n_m) \\
\nu(1,1) & \cdots &\nu(1,n_1) & 
\nu(2,1) & \cdots &\nu(2,n_2) & \cdots & 
\nu(m,1) & \cdots &\nu(m,n_m)  
\end{bmatrix}.
\end{equation*}
Hence $T^\eta$ is an expansion of $T$.
We will compare $T^\eta$ and $T'$.
Put
$$
F_k = \{ (i,j) \mid \nu(i,j) \prec \nu'(k) \}, \qquad 
k=1,\dots,m'.
$$
Since
$|\nu(i,j)|\ge K'$, 
$|\mu(i,j)|\ge L'$, 
one has 
$$
\nu'(k) = \cup_{(i,j)\in F_k}\nu(i,j).
$$
Since
$|\nu(i,j)| \ge |\nu'(k)|$,
there exist
$\eta'(k,(i,j)) \in B_*(X_A)$
such that
$$
\nu(i,j) = \nu'(k)\eta'(k,(i,j)) \quad
\text{ for }
\quad
(i,j) \in F_k.
$$
As $\tau = \tau'$, we have 
\begin{equation*}
\tau(\chi_{U_{\nu(i,j)}}) 
=\chi_{U_{\mu(i)\eta(i,j)}}
=\chi_{U_{\mu(i,j)}} 
=\tau'(\chi_{U_{\nu(i,j}}) 
=\chi_{U_{\mu'(k)\eta'(k,(i,j))}} 
\end{equation*}
so that
$$
\mu(i,j)  
=\mu'(k)\eta'(k,(i,j)) 
\quad
\text{ for }
\quad
(i,j) \in F_k.
$$
This implies that 
$T^\eta$ is an expansion of $T'$
to prove  that
$T$ is equivalent to $T'$.
\end{proof}
We denote by 
$[T]$ the equivalence class of an $A$-adic table $T$.
For $\tau \in \Gamma_A$,
denote by $T_\tau$ an $A$-adic table representing $\tau$.
The preceding lemma says that 
its equivalence class 
$[T_\tau]$ does not depend on the choice of
$T_\tau$ representing $\tau$. 
An $A$-adic table
\begin{math}
\bigl[
\begin{smallmatrix}
\mu(1) & \mu(2) & \cdots & \mu(m) \\
\nu(1) & \nu(2) & \cdots & \nu(m) 
\end{smallmatrix}
\bigr]
\end{math}
presenting $\tau \in \Gamma_A$
is said to be {\it reduced\/}
if it has a minimal length $m$ 
in the set of $A$-adic tables presenting $\tau$.
Recall that for a word $\mu=(\mu_1,\dots,\mu_k) \in B_*(X_A),$
we write 
$\Gamma_1^+(\mu) =\{j\in \{1,\dots,N\} 
\mid A(\mu_k,j)=1 \}.
$
The following lemma is obvious.
\begin{lemma}
For an $A$-adic table
\begin{math}
T=
\bigl[
\begin{smallmatrix}
\mu(1) & \mu(2) & \cdots & \mu(m) \\
\nu(1) & \nu(2) & \cdots & \nu(m) 
\end{smallmatrix}
\bigr]
\end{math}
and $i=1,\dots,m$,
let
$\Gamma_1^+(\mu(i)) =\{ \alpha_{i_1},\alpha_{i_2},\dots,\alpha_{i_{n_i}}\}$
such that
$
 \alpha_{i_1} < \alpha_{i_2} <\cdots < \alpha_{i_{n_i}}.
$
Put the words
\begin{align*}
\mu(i,1) & =\mu(i)\alpha_{i_1}, \quad
\mu(i,2)  =\mu(i)\alpha_{i_2}, \quad
\dots, \quad  
\mu(i,{n_i})  =\mu(i)\alpha_{i_{n_i}},\\
\nu(i,1) & =\nu(i)\alpha_{i_1}, \quad
\nu(i,2)  =\nu(i)\alpha_{i_2}, \quad
\dots, \quad  
\nu(i,{n_i})  =\nu(i)\alpha_{i_{n_i}}.
\end{align*}
Then the $A$-adic table $T'_i$ obtained from $T$ by replacing 
$\mu(i)$ with $\mu(i,1), \dots, \mu(i,n_i)$,
and
$\nu(i)$ with $\nu(i,1), \dots, \nu(i,n_i)$
such that  
\begin{equation*}
T'_i =
\begin{bmatrix}
\mu(1)   & \cdots &\mu(i-1)   &
\mu(i,1) & \cdots &\mu(i,n_i) & 
\mu(i+1) & \cdots &\mu(m) \\
\nu(1)   & \cdots &\nu(i-1)   &
\nu(i,1) & \cdots &\nu(i,n_i) & 
\nu(i+1) & \cdots &\nu(m) 
\end{bmatrix}
\end{equation*}
is equivalent to $T$.
 \end{lemma}
For an $A$-adic table
\begin{math}
T=
\bigl[
\begin{smallmatrix}
\mu(1) & \mu(2) & \cdots & \mu(m) \\
\nu(1) & \nu(2) & \cdots & \nu(m) 
\end{smallmatrix}
\bigr]
\end{math},
define the range depth $R(T)$ and the domain depth $D(T)$ by
\begin{equation*}
R(T)  =\Max\{ |\mu(i)| \mid {1\le i\le m} \}, \qquad
D(T)  =\Max\{ |\nu(i)| \mid {1\le i\le m} \}.
\end{equation*}
By using the above lemma recursively, we know the following lemma. 
\begin{lemma}
Let
\begin{math}
T=
\bigl[
\begin{smallmatrix}
\mu(1) & \mu(2) & \cdots & \mu(m) \\
\nu(1) & \nu(2) & \cdots & \nu(m) 
\end{smallmatrix}
\bigr]
\end{math}
be an $A$-adic table.
\begin{enumerate}
\renewcommand{\theenumi}{\roman{enumi}}
\renewcommand{\labelenumi}{\textup{(\theenumi)}}
\item
For a positive integer $M \ge D(T)$, 
there exists an $A$-adic table 
\begin{math}
T'
=\bigl[
\begin{smallmatrix}
\mu'(1) & \mu'(2) & \cdots & \mu'(m') \\
\nu'(1) & \nu'(2) & \cdots & \nu'(m') 
\end{smallmatrix}
\bigr]
\end{math}
such that 
$T' \approx T$ 
and 
$\{ \nu'(i) \mid i=1,\dots,m' \} = B_M(X_A)$.
\item
For a positive  integer $M \ge R(T)$, 
there exists an $A$-adic table 
\begin{math}
T''=
\bigl[
\begin{smallmatrix}
\mu''(1) & \mu''(2) & \cdots & \mu''(m'') \\
\nu''(1) & \nu''(2) & \cdots & \nu''(m'') 
\end{smallmatrix}
\bigr]
\end{math}
such that 
$T''\approx T$ 
and 
$\{ \mu''(i) \mid i=1,\dots,m'' \} = B_M(X_A)$.
\end{enumerate}
\end{lemma}
Let $T_1, T_2$ be two $A$-adic tables.
Take $M$ such that
$M \ge D(T_1), R(T_2)$.
By the preceding lemma,
there exist $A$-adic tables
\begin{equation*}
T'_1 =\begin{bmatrix}
\mu'_1(1) & \mu'_1(2) & \cdots & \mu'_1(p) \\
\nu'_1(1) & \nu'_1(2) & \cdots & \nu'_1(p)  
\end{bmatrix}, 
\qquad
T'_2 =\begin{bmatrix}
\mu'_2(1) & \mu'_2(2) & \cdots & \mu'_2(q) \\
\nu'_2(1) & \nu'_2(2) & \cdots & \nu'_2(q)  
\end{bmatrix}
\end{equation*}
such that
$T'_1 \approx T_1$ and 
$T'_2 \approx T_2$ and 
\begin{equation*}
|\nu'_1(1)| = \cdots = |\nu'_1(p)| 
=  
|\mu'_2(1)| = \cdots = |\mu'_2(q)|
=M.  
\end{equation*}
Hence we have
$p= q=|B_M(X_A)|$.
One may reorder $\nu'_1(i), \mu'_2(i)$ 
such as 
\begin{equation*}
\nu'_1(1)  \prec \nu'_1(2) \prec \cdots \prec \nu'_1(p), \qquad  
\mu'_2(1)  \prec \mu'_2(2) \prec \cdots \prec \mu'_2(q) 
\end{equation*}
 so that  
\begin{equation*}
\nu'_1(1) =\mu'_2(1),\quad
\nu'_1(2) =\mu'_2(2),\quad
\dots,\quad
\nu'_1(p) = \mu'_2(q). 
\end{equation*}
 Define the product 
$T'_1 \circ T'_2$ by the $A$-adic table
\begin{equation*}
T'_1 \circ T'_2
=\begin{bmatrix}
\mu'_1(1) & \mu'_1(2) & \cdots & \mu'_1(p) \\
\nu'_2(1) & \nu'_2(2) & \cdots & \nu'_2(p)  
\end{bmatrix}.
\end{equation*}
 It is easy to see that 
$T'_1 \circ T'_2$ is an $A$-adic table.
It is straightforward  to see that 
 the equivalence class
$[T'_1 \circ T'_2]$ 
does not depend on the choice of 
representatives 
$T'_1$ of $[T'_1]$
and 
$T'_2$ of $[T'_2]$.
Hence one may
define the product
$[T_1]\circ [T_2]$ by 
the equivalence class
$[T'_1 \circ T'_2]$
of the product
$T'_1 \circ T'_2$.

For an $A$-adic table
\begin{math}
T=
\bigl[
\begin{smallmatrix}
\mu(1) & \mu(2) & \cdots & \mu(m) \\
\nu(1) & \nu(2) & \cdots & \nu(m) 
\end{smallmatrix}
\bigr]
\end{math},
define an $A$-adic table
\begin{equation*}
T^{-1}
=\begin{bmatrix}
\nu(1) & \nu(2) & \cdots & \nu(m) \\
\mu(1) & \mu(2) & \cdots & \mu(m) 
\end{bmatrix}.
\end{equation*}
The identity table denoted by $I$ 
is defined by
\begin{equation*}
I
=\begin{bmatrix}
1 & 2 & \cdots & N \\
1 & 2 & \cdots & N  
\end{bmatrix}
\end{equation*}
 where
 the two rows of $I$ denote the list of the ordered 
 symbols $\{1,2,\dots,N\}=B_1(X_A)$.
\begin{lemma} Keep the above notations.
\begin{enumerate}
\renewcommand{\theenumi}{\roman{enumi}}
\renewcommand{\labelenumi}{\textup{(\theenumi)}}
\item
The equivalence class $[I]$ of $I$ is the unit of the product operations 
in the equivalence classes of the $A$-adic tables.   
\item
If $T \approx T'$, then $T^{-1} \approx {T'}^{-1}$.
\end{enumerate}
\end{lemma}
Since
$T^{-1}\circ T \approx I$ and
$T\circ T^{-1} \approx I$,
the class $[T^{-1}]$ of $T^{-1}$ is the inverse of $[T]$
in the equivalence classes of the $A$-adic tables.


\begin{definition}\label{defn:tablegroup}
Denote by
$
\Gamma_A^{\tab}
$
the group of the equivalence classes of $A$-adic tables.
\end{definition}

Therefore we have
\begin{proposition}\label{prop:tab}
The correspondence
$\tau \in \Gamma_A \longrightarrow [T_\tau] \in \Gamma_A^{\tab}$
gives rise to an isomorphism of groups.
\end{proposition}
\begin{proof}
Let $\tau, \tau' \in \Gamma_A$.
By Lemma \ref{lem:tauT},
$\tau = \tau'$ if and only if 
$[T_\tau] = [T_{\tau'}]$.
It is direct to see that 
for $\tau_1, \tau_2 \in \Gamma_A$,
the equivalence class 
$[T_{\tau_1 \circ \tau_2}]$
of an $A$-adic table 
$T_{\tau_1 \circ \tau_2}$
representing the composition
$\tau_1 \circ \tau_2$
 is the product
$[T_{\tau_1}]\circ [T_{\tau_2}]$
of the classes
$[T_{\tau_1}], [T_{\tau_2}]$.
Hence
the correspondence
$\tau \in \Gamma_A \longrightarrow [T_\tau] \in \Gamma_A^{\tab}$
gives rise to an isomorphism of groups.
\end{proof}

\section{Isomorphisms among  $\Gamma_A$, $\Gamma_A^{\tab}$ 
and $\Gamma_A^{\PL}$}
In the preceding section,
we have shown that the two groups
$\Gamma_A$, $\Gamma_A^{\tab}$ are isomorphic.
In this section, we will show that these two groups are 
isomorphic to the group $\Gamma_A^{\PL}$
of $A$-adic PL functions. 
\begin{lemma}\label{lem:tableSFT}
For an  $A$-adic table 
\begin{math}
T=
\bigl[
\begin{smallmatrix}
\mu(1) & \mu(2) & \cdots & \mu(m) \\
\nu(1) & \nu(2) & \cdots & \nu(m) 
\end{smallmatrix}
\bigr]
\end{math},
there exist an $A$-adic pattern of rectangles 
whose rectangle slopes are
\begin{equation*}
\beta^{|\nu(1)| -|\mu(1)|}, 
\beta^{|\nu(2)| -|\mu(2)|},
\dots, 
\beta^{|\nu(m)| -|\mu(m)|},
\end{equation*}
and an $A$-adic PL function $f_T$ 
having these rectangle slopes such that 
\begin{equation}
f_T(I_{\nu(i)}) = I_{\mu(i)}, \qquad i=1,2,\dots, m. \label{eq:ft}
\end{equation}

Conversely, 
for an $A$-adic PL function $f$ with the $A$-adic pattern of rectangles 
$
I_p \times J_{\sigma(p)},  p=1,2,\dots,m 
$
 and a permutation $\sigma$
on $\{1,\dots,m\}$, 
there exists an $A$-adic table
\begin{math}
T_f=
\bigl[
\begin{smallmatrix}
\mu(1) & \mu(2) & \cdots & \mu(m) \\
\nu(1) & \nu(2) & \cdots & \nu(m) 
\end{smallmatrix}
\bigr]
\end{math}
such that 
\begin{equation*}
I_p = I_{\nu(p)}, \qquad
J_{\sigma(p)} = I_{\mu(p)}, \qquad p=1,2,\dots, m.
\end{equation*}
\end{lemma}
\begin{proof}
We are assuming the ordering such as
$\nu(1) \prec\dots \prec\nu(m)$.
Since
$X_A$ is a disjoint union $X_A = \sqcup_{j=1}^m U_{\mu(j)}$,
there exists a permutation $\sigma_0$ on $\{1,2,\dots,m\}$
such that
$
\mu(\sigma_0(1)) \prec 
\mu(\sigma_0(2)) \prec
\dots \prec
\mu(\sigma_0(m))$.
Put
\begin{equation*}
x_i = l(\nu(i+1)), 
\qquad 
y_i = l(\mu(\sigma_0(i+1))), \qquad i= 0, 1,\dots,m-1
\end{equation*}
so that $x_0 = y_0 =0$ 
and
\begin{equation*}
I_p =[x_{p-1}, x_p), \qquad 
J_p =[y_{p-1}, y_p), 
 \qquad p=1,2,\dots,m
\end{equation*}
where 
$x_m = y_m =1$.
Define the permutation 
$\sigma := \sigma_0^{-1}$ on $\{1,2,\dots,m\}$.
We note that
$r(\nu(i)) = l(\nu(i+1)),
r(\mu(\sigma_0(i))) = l(\mu(\sigma_0(i+1)))
$
for
$
i=1,\dots,m-1$.
Then the rectangles
$I_p \times J_{\sigma(p)}, p=1,2,\dots,m$
are $A$-adic rectangles 
by Lemma \ref{lem:five}
such that
\begin{equation*}
\frac{y_{\sigma(p)} -y_{\sigma(p)-1}}{x_p -x_{p-1}}
=
\frac{r(\mu(p)) -l(\mu(p))}{r(\nu(p)) -l(\nu(p))}.
\end{equation*}
We then have
\begin{equation*}
r(\nu(p)) - l(\nu(p)) 
= \varphi(S_{\nu(p)}S_{\nu(p)}^*)
= \frac{1}{\beta^{ |\nu(p)|}} \varphi(S_{\nu(p)}^*S_{\nu(p)})
\end{equation*}
and similarly
$
r(\mu(p)) - l(\mu(p)) 
 = \frac{1}{\beta^{|\mu(p)|}} \varphi(S_{\mu(p)}^*S_{\mu(p)}).
$
As the condition
$\Gamma^+_*(\nu(p)) = \Gamma^+_*(\mu(p))$ 
implies 
$S_{\nu(p)}^*S_{\nu(p)}= S_{\mu(p)}^*S_{\mu(p)}$,
we have
\begin{equation*}
\frac{y_{\sigma(p)} -y_{\sigma(p)-1}}{x_p -x_{p-1}}
=
\beta^{|\nu(p)| - |\mu(p)|},
\qquad p=1,2,\dots,m.
\end{equation*}
By Proposition \ref{prop:rectanglePL},
one immediately knows that
the associated $A$-adic PL function denoted by $f_T$ with the above 
$A$-adic pattern of rectangles
satisfies the condition \eqref{eq:ft}.

The converse implication is straightforward
from Lemma \ref{lem:five}.
\end{proof}

We may directly construct an $A$-adic PL function $f_T$ 
from an $A$-adic table 
\begin{math}
T=
\bigl[
\begin{smallmatrix}
\mu(1) & \mu(2) & \cdots & \mu(m) \\
\nu(1) & \nu(2) & \cdots & \nu(m) 
\end{smallmatrix}
\bigr]
\end{math}
as follows.
Put
$x_i = l(\nu(i+1)), \hat{y}_i = l(\mu(i+1))$
and
$f_T(x_i) = \hat{y}_i, i=0, 1,\dots,m-1$.
Define
$f_T(x)$  on $[x_{i-1},x_i)$
as a linear function
with slope
$ \beta^{|\nu(i)| -|\mu(i)|}
(=\frac{r(\mu(i)) - l(\mu(i))}{r(\nu(i)) - l(\nu(i))} 
=\frac{\hat{y}_i - \hat{y}_{i-1}}{x_i - x_{i-1}})
$
for $i=1,2, \dots,m$.
It is easy to see that the function $f_T$ 
is an $A$-adic PL function.
Let us denote by
$\iota$ the $A$-adic PL function defined by
$\iota(x) = x, x \in [0,1)$.
The following lemma is direct.
\begin{lemma}\label{lem:PLT}
For  two $A$-adic tables $T_1, T_2$, we have 
\begin{enumerate}
\renewcommand{\theenumi}{\roman{enumi}}
\renewcommand{\labelenumi}{\textup{(\theenumi)}}
\item
$T_1$ is equivalent to $T_2$ if and only if
$f_{T_1} = f_{T_2}$ as functions.
Hence we may write $f_T$ as $f_{[T]}$.  
\item
$f_{[T_1]\circ [T_2]} = f_{[T_1]} \circ f_{[T_2]}$.
\item $\iota = f_{[I]}$.
\end{enumerate}
\end{lemma}
We reach the main result of the paper.
\begin{theorem}\label{thm:SFTPL}
There exist canonical isomorphisms of discrete groups among 
the continuous full group $\Gamma_A$, 
the group $\Gamma^{\tab}_A$ of the equivalence classes  of $A$-adic tables,
and
the group $\Gamma^{\PL}_A$ of $A$-adic PL functions on $[0,1)$,
that is 
\begin{equation*}
\Gamma_A\cong \Gamma^{\tab}_A \cong \Gamma^{\PL}_A.
\end{equation*}
In particular,
the continuous full group
$\Gamma_A$
for a topological Markov shift $(X_A,\sigma_A)$
is realized as the group of all $A$-adic PL functions on $[0,1)$.
\end{theorem}
\begin{proof}
By Proposition \ref{prop:tab},
 we have an isomorphism
from the continuous full group
$\Gamma_A$ 
to the group 
$\Gamma^{\tab}_A$ of the equivalence classes of 
$A$-adic tables.
By Lemma \ref{lem:tableSFT} and Lemma \ref{lem:PLT},
the correspondence
$
[T] \in \Gamma^{\tab}_A
\longrightarrow 
 f_T \in \Gamma^{\PL}_A
$ 
yields 
an isomorphism.
\end{proof}



\section{A realization of $\Gamma_A$ as $A$-adic PL functions}
In this section, we will construct a continuous surjection  
of the shift space $X_A$
onto the interval $[0,1]$ 
which yields a representation of elements of 
the continuous full group $\Gamma_A$ 
to the group $\Gamma_A^{\PL}$ of $A$-adic PL functions.
For $x = (x_i)_{i \in {\mathbb{N}}} \in X_A$ and $n \in \Zp$, 
consider the word $(x_1,\dots,x_n) \in B_n(X_A)$
and set 
\begin{equation*}
l_n(x) = l(x_1,\dots,x_n), \qquad r_n(x) = r(x_1,\dots,x_n).
\end{equation*}
\begin{lemma}
For $x = (x_i)_{i \in {\mathbb{N}}} \in X_A$ and $n \in \Zp$, we have
\begin{enumerate}
\renewcommand{\theenumi}{\roman{enumi}}
\renewcommand{\labelenumi}{\textup{(\theenumi)}}
\item $l_n(x) \le l_{n+1}(x) \le r_{n+1}(x) \le r_n(x).$
\item $| r_n(x) - l_n(x) | \le \frac{1}{\beta^n}$.
\end{enumerate}
\end{lemma}
\begin{proof}
(i)
For 
$\mu =(\mu_1,\dots,\mu_n) \in B_n(X_A)$,
the condition $\mu \prec (x_1,\dots,x_n)$ 
implies 
$\mu j \prec (x_1,\dots,x_n,x_{n+1})$ for all $j $ with $A(\mu_n,j) =1$
so that 
\begin{align*}
l_n(x)
& = 
\sum_{\substack{
\mu \in B_n(X_A)\\
 \mu \prec (x_1,\dots,x_n)
 }} 
\varphi(S_\mu S_\mu^*) 
= 
\sum_{j=1}^N A(\mu_n, j)
\sum_{\substack{
\mu \in B_n(X_A)\\
 \mu\prec (x_1,\dots,x_n)
 }} \varphi(S_{\mu j} S_{\mu j}^*) \\
& \le 
\sum_{\substack{
\nu \in B_{n+1}(X_A)\\
 \nu \prec (x_1,\dots,x_n,x_{n+1})
 }} 
 \varphi(S_\nu S_\nu^*)
 =  l_{n+1}(x).
\end{align*} 
We note that 
\begin{equation}
l_{n+1}(x)=
l_n(x) + \sum_{j< x_{n+1}} 
\varphi(S_{x_1\cdots x_n j} S_{x_1\cdots x_n j}^*) \label{eq:ln}
\end{equation}
so that
\begin{align*}
r_{n+1}(x)
& = l_{n+1}(x) + 
\varphi(S_{x_1\cdots x_n x_{n+1}} S_{x_1\cdots x_n x_{n+1}}^*)\\
& =l_n(x) + \sum_{j \le x_{n+1}} 
\varphi(S_{x_1\cdots x_n j} S_{x_1\cdots x_n j}^*) \\
& \le l_n(x) + \sum_{j=1}^N 
\varphi(S_{x_1\cdots x_n j} S_{x_1\cdots x_n j}^*) \\
& =l_n(x) +  \varphi(S_{x_1\cdots x_n} S_{x_1\cdots x_n}^*) = r_n(x).
\end{align*}

(ii) By the equality
$r_n(x) = l_n(x) + \varphi(S_{x_1\cdots x_n} S_{x_1\cdots x_n}^*)$
with
\begin{equation*}
\varphi(S_{x_1\cdots x_n} S_{x_1\cdots x_n}^*)=
\frac{1}{\beta^n}\sum_{j=1}^N A(x_n,j)p_j, \qquad \sum_{j=1}^N p_j = 1,
\end{equation*}
we have $| r_n(x) - l_n(x)| \le \frac{1}{\beta^n}.$
\end{proof}    
\begin{lemma}\label{lem:minmax}
For $x = (x_i)_{i \in {\mathbb{N}}} \in X_A$ and $n \in \Zp$, 
we have
\begin{enumerate}
\renewcommand{\theenumi}{\roman{enumi}}
\renewcommand{\labelenumi}{\textup{(\theenumi)}}
\item $l_n(x) = l_{n+1}(x)$ if and only if
$x_{n+1} = \Min \{j=1,\dots,N \mid A(x_n,j) =1 \}$.
\item $r_n(x) = r_{n+1}(x)$ if and only if
$x_{n+1} = \Max \{j=1,\dots,N \mid A(x_n,j) =1 \}$.
\end{enumerate}
\end{lemma}
\begin{proof}
(i) 
By
\eqref{eq:ln},
one sees that 
$
l_{n+1}(x)=
l_n(x)
$
if and only if
$ 
\sum_{j< x_{n+1}} 
\varphi(S_{x_1\cdots x_n j} S_{x_1\cdots x_n j}^*)
=0.
$
Since the state $\varphi$ on $\DA$ is faithful,
the latter condition is equivalent to
the condition that there does not exist any
$j=1,\dots,N$ such that 
$j < x_{n+1}$ and 
$A(x_n,j) =1$.
Hence we have
the desired assertion.

(ii) is similar to (i).
\end{proof}
For a word $\omega = (\omega_1,\dots,\omega_n) \in B_n(X_A)$,
let us denote by 
$\omega_{\min} =(\underline{\omega}_i)_{i \in \N}\in X_A$
(resp. $\omega_{\max} =(\overline{\omega}_i)_{i \in \N} \in X_A$)
its minimal (resp. maximal) 
extension to a right infinite sequence in $X_A$,
which is defined  by setting
\begin{align*}
\underline{\omega}_i
& = \omega_i \quad (\text{resp. } \overline{\omega}_i = \omega_i) \quad
\text{ for } i=1,\dots,n,\\
\underline{\omega}_{n+k}
& = \Min\{j=1,2,\dots,N \mid A(\underline{\omega}_{n+k-1},j) =1\}, \\
(\text{resp. } \overline{\omega}_{n+k}
& = \Max\{j=1,2,\dots,N \mid A(\overline{\omega}_{n+k-1},j) =1\})
\quad \text{ for } k=1,2,\dots.
\end{align*}
By Lemma \ref{lem:minmax},
one has 
$l(\omega) = l_{n+k}(\omega_{\min})$
and
$r(\omega) = r_{n+k}(\omega_{\max})$
 for all $k \in \N$. 
For the two symbols 
$1, \, N \in B_1(X_A)$,
we may consider  
the elements $1_{\min}, \, N_{\max}$ in $X_A$
so that we see
\begin{lemma}
$l_n(1_{\min}) = 0, \,  r_n(N_{\max}) =1$ 
for all $n \in \N$.
\end{lemma}
For two sequences 
$x=(x_n)_{n \in \N}, y=(y_n)_{n \in \N} \in X_A$,
we write 
$x \prec y$ 
if $x_1 = y_1, \dots, x_n = y_n, x_{n+1} < y_{n+1}$
for some $n \in \Zp$.
Hence $X_A$ becomes an ordered space such that 
$1_{\min}$ (resp. $N_{\max}$)  
is minimum (resp. maximum).
Recall that for a word $\mu \in B_*(X_A)$, denote by
$I_\mu$ the interval $[l(\mu),r(\mu))$, so that 
$\bar{I}_\mu = [l(\mu),r(\mu)]$.
\begin{proposition} 
There exists an order preserving surjective continuous map
$\rho_A: X_A \longrightarrow [0,1]$ such that 
\begin{equation*}
\rho_A(1_{\min}) =0,\qquad
\rho_A(N_{\max}) = 1
\quad
\text{ and }
\quad
\rho_A(U_\mu) = \bar{I}_\mu 
\quad
\text{ for } \mu \in B_n(X_A).
\end{equation*}
\end{proposition}
\begin{proof}
For $x = (x_i)_{i \in {\mathbb{N}}} \in X_A$, 
there exists an element 
$
\lim_{n\to\infty}l_n(x) (=\lim_{n\to\infty}r_n(x))
$
in $[0,1]$ 
which we denote by
$\rho_A(x)$. 
It satisfies the inequalities
$l_n(x) \le \rho_A(x) \le r_n(x)$ for all $n \in {\mathbb{N}}$.
By the above lemma,
we have
\begin{equation*}
\rho_A(1_{\min})=
\lim_{n \to\infty}l_n(1_{\min}) =0,
\qquad
\rho_A(N_{\max})=
\lim_{n \to\infty}r_n(N_{\max}) =1.
\end{equation*}
We will next
show that $\rho_A:X_A\longrightarrow [0,1]$
is surjective.
For $t \in [0,1]$,
we may assume that $t <1$
because $\rho_A(N_{\max}) =1$.
For $n \in\N$,
by Lemma \ref{lem:Imu} (ii),
one may find a word $\mu^{(n)} \in B_n(X_A)$
such that 
$t \in I_{\mu^{(n)}}$.
The first $n$-symbols of $\mu^{(n+1)}$
coincide with $\mu^{(n)}$ 
so that the sequence
$\{\mu^{(n)}\}_{ n \in \N}$ 
of words defines a right infinite sequence
$x_t =(x_n)_{n \in \N}$ of $X_A$
such that 
$(x_1,\dots,x_n) =\mu^{(n)}$.
Since 
$l(\mu^{(n)}) \le t \le r(\mu^{(n)})$ 
and
$| r(\mu^{(n)}) - l(\mu^{(n)}) | <\frac{1}{\beta^n}$,
one sees that
$\rho_A(x_t) = \lim_{n\to\infty} l(\mu^{(n)}) = t$
so that $\rho_A:X_A\longrightarrow [0,1]$
is surjective.

 For $\mu \in B_n(X_A)$ and $x \in U_\mu$, 
 one sees that
$l(\mu) =l_n(x) \le \rho_A(x) \le r_n(x) = r(\mu)$ so that 
$\rho_A(x) \in [l(\mu), r(\mu)]$. 
Hence we have
$\rho_A(U_\mu) \subset \bar{I}_\mu$.
As $\rho_A(X_A) = [0,1]$ 
and
$[0,1) =\sqcup_{\mu \in B_n(X_A)}I_\mu$ is a disjoint union
for a fixed $n \in \N$,
one has 
$I_\mu \subset \rho_A(U_\mu)$ 
so that 
$\rho_A(U_\mu) = \bar{I}_\mu$.
This also shows that $\rho_A$ is order preserving.
%
%
%
%
%
%
\end{proof}

We will represent $A$-adic PL functions on $[0,1] $
by using the surjection $\rho_A:X_A\longrightarrow [0,1]$.
For $\tau \in \Gamma_A$,
let
\begin{math}
T_\tau=
\bigl[
\begin{smallmatrix}
\mu(1) & \mu(2) & \cdots & \mu(m) \\
\nu(1) & \nu(2) & \cdots & \nu(m) 
\end{smallmatrix}
\bigr]
\end{math}
be its reduced representation.
Let $C_\tau$ be the finite subset of $[0,1]$
defined by 
\begin{equation*}
C_\tau =\{ l(\nu(i)) \mid i=2,3,\dots,m\}
(=\{ r(\nu(i)) \mid i=1,2,\dots,m-1\}).
\end{equation*}
Then the $A$-adic PL function $f_\tau$ 
associated with the $A$-adic table
$T_\tau$
is continuous and linear
on $[0,1)$ except $C_\tau$.
We define
a finite subset $S_\tau$ of $X_A$ by
\begin{equation*}
S_\tau =\{ \nu(i)_{\min} \in X_A
\mid i=1,2,\dots,m\}
\end{equation*}
so that
$\rho_A(S_\tau) = C_\tau$.
\begin{proposition}
For $\tau \in \Gamma$, we have 
$f_\tau(\rho_A(x) ) = \rho_A(\tau(x))$ for all 
$x \in X_A \backslash S_\tau$.
\end{proposition}
\begin{proof}
Since 
$X_A $ is a disjoint union $\sqcup_{i=1}^M U_{\nu(i)}$, 
for $x \in X_A \backslash S_\tau$
we may take $\nu(i)=(\nu(i)_1,\dots,\nu(i)_{l_i})$ 
such that
$x \in U_{\nu(i)}$.
We write
$x = (\nu(i)_1,\dots,\nu(i)_{l_i},x_{l_i+1},x_{l_i+2},\dots)$.
As $x \not\in S_\tau$, the function
$f_\tau$ is continuous at $x$. 
It then follows that 
\begin{align*}
f_\tau(\rho_A(x) ) 
& = f_\tau(
\lim_{n\to\infty} r(\nu(i)_1,\dots,\nu(i)_{l_i},
     x_{l_i+1},\dots,x_{l_i+n})) \\
& = \lim_{n\to\infty} 
    f_\tau(
 r(\nu(i)_1,\dots,\nu(i)_{l_i},
     x_{l_i+1},\dots,x_{l_i+n})) \\
& = \lim_{n\to\infty} 
 r(\mu(i)_1,\dots,\mu(i)_{k_i},
     x_{l_i+1},\dots,x_{l_i+n}) \\
& = \rho_A(\tau(x)).
\end{align*}
\end{proof}
We will next define the derivative of $\tau \in \Gamma_A$.
For $\tau \in \Gamma_A$,
let $l_\tau, k_\tau $
be $\Zp$-valued 
continuous functions on $X_A$
satisfying \eqref{eq:tau}.
\begin{lemma}
For $\tau \in \Gamma_A$,
define
$d_\tau:X_A \longrightarrow {\mathbb{Z}}$
by setting
\begin{equation*}
d_\tau(x) = l_\tau(x)- k_\tau(x), \qquad x \in X_A.
\end{equation*}
Then $d_\tau$ does not depend on the choice of 
the functions
$l_\tau, k_\tau$
satisfying
\eqref{eq:tau}.
\end{lemma}
\begin{proof}
Let $l'_\tau, k'_\tau: X_A\rightarrow \Zp $
be another continuous functions
such that
\begin{equation}
\sigma_A^{k'_\tau(x)}(\tau(x)) = 
\sigma_A^{l'_\tau(x)}(x), \qquad x \in X_A. \label{eq:tau'}
\end{equation}
For $x = (x_i)_{i \in \N} \in X_A$,
the identities
\eqref{eq:tau} and \eqref{eq:tau'}
ensure us that 
there exist words
$(\mu_1(x), \dots,\mu_{k_{\tau}(x)}(x))\in B_{{k_{\tau}(x)}}(X_A)$
and
$(\mu'_1(x),\dots,\mu'_{k'_{\tau}(x)}(x)) \in B_{{k'_{\tau}}(x)}(X_A)$
such that 
\begin{align*}
\tau(x) 
& = (\mu_1(x), \dots,\mu_{k_{\tau}(x)}(x), x_{l_{\tau}(x) +1},x_{l_{\tau}(x) +2},\dots )\\
& = (\mu'_1(x),\dots,\mu'_{k'_{\tau}(x)}(x), x_{l'_{\tau}(x) +1},x_{l'_{\tau}(x) +2},\dots ).
\end{align*}
For any $n > k_{\tau}(x), k'_{\tau}(x)$, 
by taking the $n$th coordinates of the above sequences,
we see that 
\begin{equation*}
x_{n-k_{\tau}(x) + l_{\tau}(x)} =x_{n-k'_{\tau}(x) + l'_{\tau}(x)}. 
\end{equation*}
Put
$d'_\tau(x) = l'_\tau(x)- k'_\tau(x)
$
and
$K(x) = \Max\{ k_{\tau}(x), k'_{\tau}(x)\}$,
so that
\begin{equation*}
\sigma_A^{K(x) + d_{\tau}(x)}(x)
 = \sigma_A^{K(x) +d'_{\tau}(x)}(x).
\end{equation*}
Suppose that
$d_\tau(x) \ne d'_{\tau}(x)$
for some $ x \in X_A$.
The above equality implies that 
$x$ is an eventually periodic point.
As the functions 
$K, d_\tau, d_{\tau'}$ are all continuous,
all elements of some neighborhood of $x$
are eventually periodic.
Since the set of non-eventually periodic points is dense in $X_A$,
we have a contradiction and hence 
$d_\tau = d'_\tau$.
\end{proof}

\begin{lemma}
For $\tau, \tau_1,\tau_2 \in \Gamma_A$, we have  
\begin{enumerate}
\renewcommand{\theenumi}{\roman{enumi}}
\renewcommand{\labelenumi}{\textup{(\theenumi)}}
\item $ d_{\tau_2\circ \tau_1} =d_{\tau_1} + d_{\tau_2} \circ \tau_1.$
\item $ d_{\tau^{-1}} = - d_{\tau} \circ \tau^{-1}.$ 
 \end{enumerate}
\end{lemma}
\begin{proof}
(i) For $\tau_i \in \Gamma_A$,
take continuous functions $k_{\tau_i},l_{\tau_i}: X_A \longrightarrow \Zp$
such that 
\begin{equation*}
\sigma_A^{k_{\tau_i}(x)}(\tau_i(x)) 
= \sigma_A^{l_{\tau_i}(x)}(x), \qquad  i=1,2, \, x \in X_A
\end{equation*}
so that
\begin{equation*}
\sigma_A^{k_{\tau_2}(\tau_1(x))}(\tau_2(\tau_1(x))) 
= \sigma_A^{l_{\tau_2}(\tau_1(x))}(\tau_1(x)), \qquad x \in X_A. 
\end{equation*}
It then follows that
\begin{equation*}
\sigma_A^{k_{\tau_1}(x)}(\sigma_A^{k_{\tau_2}(\tau_1(x))}(\tau_2(\tau_1(x)))) 
 = \sigma_A^{l_{\tau_2}(\tau_1(x))}(\sigma_A^{k_{\tau_1}(x)}(\tau_1(x))) 
= \sigma_A^{l_{\tau_2}(\tau_1(x))}(\sigma_A^{l_{\tau_1}(x)}(x))
\end{equation*}
so that
\begin{equation*}
\sigma_A^{k_{\tau_1}(x) + k_{\tau_2}(\tau_1(x))}(\tau_2\circ\tau_1(x)) 
 = \sigma_A^{l_{\tau_1}(x) + l_{\tau_2}(\tau_1(x))}(x).
\end{equation*}
Hence we have
\begin{equation*}
d_{\tau_2 \circ \tau_1}(x) 
=\{ l_{\tau_1}(x) + l_{\tau_2}(\tau_1(x)) \}
-\{k_{\tau_1}(x) + k_{\tau_2}(\tau_1(x)) \}
= d_{\tau_1}(x) +d_{\tau_2}(\tau_1(x)).
\end{equation*}
(ii)
By \eqref{eq:tau}, we have
\begin{equation*}
\sigma_A^{k_\tau(\tau^{-1}(x))}(x) 
= \sigma_A^{l_\tau(\tau^{-1}(x))}(\tau^{-1}(x)), \qquad x \in X_A 
\end{equation*}
so that
\begin{equation*}
d_{\tau^{-1}}(x) 
=k_{\tau}(\tau^{-1}(x))  - l_{\tau}(\tau^{-1}(x)) 
=-  d_{\tau}(\tau^{-1}(x)).
\end{equation*}
\end{proof}
 \begin{definition}\label{defn:Dtau}
 For an element $\tau \in \Gamma_A$,
the {\it derivative\/} $D_\tau$ of $\tau$ 
is defined by  a real valued continuous function
$D_\tau$ on $X_A$:
\begin{equation}
D_\tau(x) = \beta^{d_\tau(x)},\qquad x \in X_A,
\end{equation}
where $\beta$ is the Perron--Frobenius eigenvalue of the matrix $A$.
\end{definition}
The derivative  $D_\tau$ of $\tau$ 
is regarded as an element of $\DA$.
Recall that $\varphi$ stands for the continuous linear functional on $\DA$ for the
unique  probability measure on $X_A$ satisfying \eqref{eq:RPF}.
The following proposition shows that 
$D_\tau$ satisfies the law of derivatives.
\begin{proposition}\label{prop:derivative}
For $\tau, \tau_1,\tau_2 \in \Gamma_A$, we have  
 \begin{enumerate}
\renewcommand{\theenumi}{\roman{enumi}}
\renewcommand{\labelenumi}{\textup{(\theenumi)}}
\item $\varphi(D_\tau) =1.$
\item $ D_{\tau_2\circ \tau_1} = D_{\tau_1}\cdot (D_{\tau_2} \circ \tau_1).$
\item $ D_{\tau^{-1}} = (D_{\tau} \circ \tau^{-1})^{-1}.$ 
 \end{enumerate}
\end{proposition}
\begin{proof}
(i)
Suppose that $\tau$ 
is given by an $A$-adic table
\begin{math}
T=
\bigl[
\begin{smallmatrix}
\mu(1) & \mu(2) & \cdots & \mu(m) \\
\nu(1) & \nu(2) & \cdots & \nu(m) 
\end{smallmatrix}
\bigr]
\end{math}
so that
$u_\tau = \sum_{i=1}^m S_{\mu(i)}S_{\nu(i)}^*$,
$
S_{\mu(i)}^* S_{\mu(i)}  = S_{\nu(i)}^*S_{\nu(i)}
$
and
$ 
\sum_{i=1}^m S_{\mu(i)}S_{\mu(i)}^* = \sum_{i=1}^m S_{\nu(i)}S_{\nu(i)}^*=1.
$ 
Recall that the positive operator $\lambda_A:\DA \rightarrow \DA$ 
is defined by 
$\lambda_A(f) = \sum_{i=1}^N S_i^* f S_i$ for $f \in \DA$.
It then follows that
$$ 
\lambda_A^{|\mu(i)|}( u_\tau S_{\nu(i)}S_{\nu(i)}^* u_{\tau}^*)
 = \lambda_A^{|\mu(i)|}(  S_{\mu(i)}S_{\mu(i)}^*)
 = S_{\mu(i)}^* S_{\mu(i)}= S_{\nu(i)}^* S_{\nu(i)}
$$
so that
\begin{equation*}
\lambda_A^{|\mu(i)|}( u_\tau S_{\nu(i)}S_{\nu(i)}^* u_{\tau}^*)
 = \lambda_A^{|\nu(i)|}(  S_{\nu(i)}S_{\nu(i)}^*), \quad i=1,\dots,m.
\end{equation*}
As $\varphi \circ \lambda_A = \beta \varphi$ on $\DA$,
we have
\begin{equation}
\varphi( u_\tau S_{\nu(i)}S_{\nu(i)}^* u_{\tau}^*)
 = \beta^{|\nu(i)|-|\mu(i)|}
 \varphi( S_{\nu(i)}S_{\nu(i)}^*), \quad i=1,\dots,m. \label{eq:vphiut}
 \end{equation}
Since
$d_\tau(x) = l_\tau(x)- k_\tau(x) = |\nu(i)| - |\mu(i)|$ for 
$x \in U_{\nu(i)}$,
the derivative  $D_\tau$
is expressed as  
\begin{equation*}
D_\tau = \sum_{i=1}^m \beta^{|\nu(i)| - |\mu(i)|}S_{\nu(i)}S_{\nu(i)}^*
\end{equation*}
so that by the equality \eqref{eq:vphiut}
one obtains  that  
\begin{equation*}
\varphi(D_\tau)
= \sum_{i=1}^m \beta^{|\nu(i)| - |\mu(i)|}\varphi(S_{\nu(i)}S_{\nu(i)}^*) 
 = \sum_{i=1}^m \varphi( u_\tau S_{\nu(i)}S_{\nu(i)}^* u_\tau^*) 
  = \varphi(1) =1.
\end{equation*}
 
 (ii), (iii)
By the previous lemma, we have
\begin{align*}
D_{\tau_2\circ \tau_1} 
& = \beta^{d_{\tau_2\circ \tau_1}}
  = \beta^{d_{\tau_1}} \cdot \beta^{d_{\tau_2} \circ \tau_1}
  = D_{\tau_1} \cdot D_{\tau_2} \circ {\tau_1},\\
D_{\tau^{-1}} 
& = \beta^{- d_{\tau} \circ \tau^{-1}}
  = [D_{\tau} \circ \tau^{-1}]^{-1}.
\end{align*}
\end{proof}

As the function $f_{\tau}$ is linear on the interval
$I_{\nu(i)} = [l(\nu(i)), r(\nu(i)))$
with slope
$
\beta^{|\nu(i)| - |\mu(i)|},
$
 we may summarize the above discussions in the following theorem.
\begin{theorem}\label{thm:rhoA}
There exists an order preserving continuous surjection
$\rho_A: X_A \longrightarrow [0,1]$
from the shift space $X_A$ of a one-sided topological Markov shift 
$(X_A,\sigma_A)$ to the closed interval $[0,1]$  
such that
for any element $\tau \in \Gamma_A$,
there exists a
finite set $S_\tau \subset X_A$ such that 
the  corresponding $A$-adic PL function $f_\tau$ for  $\tau$
satisfies the following properties:
\begin{enumerate}
\renewcommand{\theenumi}{\roman{enumi}}
\renewcommand{\labelenumi}{\textup{(\theenumi)}}
\item
$f_\tau(\rho_A(x)) = \rho_A(\tau(x))$ for $x \in X_A\backslash S_\tau$,
\item 
$\frac{d f_\tau}{dt}(\rho_A(x)) = D_\tau(x) = \beta^{d_\tau(x)}$
for $x \in X_A\backslash S_\tau$,
\end{enumerate}
where
$d_\tau(x) = l_\tau(x) - k_\tau(x)$
for the continuous functions 
$k_\tau, l_\tau :X_A\longrightarrow \Zp$ 
satisfying
$\sigma_A^{k_\tau(x)}(\tau(x)) =\sigma_A^{l_\tau(x)}(x),
x \in X_A$
and $\beta$ is the Perron--Frobenius eigenvalue of $A$.
\end{theorem}

\section{Generalizations of other Thompson groups}
R. J. Thompson has defined  finitely presented infinite subgroups
$F_2, T_2 $ of $V_2$ 
which satisfy
$F_2 \subset T_2 \subset V_2$.
K. S. Brown \cite{Brown}
has extended the subgroups 
$F_2 , T_2$ of $V_2$
to the family
$F_N \subset T_N \subset V_N$
of finitely presented subgroups
$F_N, T_N$
of $V_N$
such that
$T_N$
is a  group of piecewise linear homeomorphisms 
$f: [0,1] \longrightarrow [0,1]$
on the unit circle
having finitely many singularities such that
all singularities of $f$ are in ${\mathbb{Z}}[\frac{1}{N}]$,
the derivative of $f$ at any non-singular point is
$N^k$ for some $k \in {\mathbb{Z}}$,
 and 
$F_N$
is a subgroup of 
$T_N$ consisting of 
piecewise linear homeomorphisms 
$f: [0,1] \longrightarrow [0,1]$
on the unit interval.

In this section,
we generalize the groups 
$F_N, T_N$ for $1 < N \in {\mathbb{N}}$
to
$F_A, T_A$ for irreducible square matrices $A$ with entries in 
$\{ 0,1 \}$ by using the techniques of the preceding sections.

Recall that an element $\tau \in \Gamma_A$
is represented as a cylinder map given by two families 
$\mu(i),\nu(i), i=1,\dots,m$
of words satisfying 
\eqref{eq:cylinder1},
\eqref{eq:cylinder2},
\eqref{eq:cylinder3}
and
\eqref{eq:cylinder4}.
We may assume that the words 
$\nu(i), i=1,\dots,m$
are ordered such as 
$
\nu(1) \prec \nu(2) \prec \cdots \prec  \nu(m).
$
We define further properties for $\tau \in \Gamma_A$
as follows.
$\tau \in \Gamma_A$ is said to be
\begin{enumerate}
\renewcommand{\theenumi}{\roman{enumi}}
\renewcommand{\labelenumi}{\textup{(\theenumi)}}
\item 
 {\it order preserving\/}\
if one may take the words 
$\mu(i), i=1,\dots,m$ such as
\begin{equation*}
\mu(1) \prec \mu(2) \prec \cdots \prec  \mu(m),
\end{equation*}
\item  
{\it cyclic order preserving\/} if 
one may take the words 
$\mu(i), i=1,\dots,m$ such as
\begin{equation*}
\mu(k) \prec \mu(k+1) \prec \cdots \prec  \mu(m)\prec
 \mu(1) \prec \mu(2)   \prec \cdots \prec  \mu(k-1)
 \end{equation*} for some $k \in \{1,2,\dots,m\}$.
\end{enumerate}
If $\tau$ is order preserving, it is cyclic order preserving.
It is easy to see that 
the set of order preserving cylinder maps forms a subgroup of 
$\Gamma_A$, and
the set of cyclic
order preserving cylinder maps forms a subgroup of $\Gamma_A$.
We denote them by $F_A$ and by $T_A$
and call them 
the order preserving continuous full group
and
the cyclic order preserving continuous full group,
 respectively.


In Definition \ref{defn:rectangle} (ii),
if one may take $\sigma$ such as 
\begin{equation}
\sigma(k) < \sigma(k+1) < 
\cdots    < \sigma(m) < 
\sigma(1) < \sigma(2) < \cdots <\sigma(k-1) \label{eq:sigma}
\end{equation}
for some $k \in \{1,\dots,m\}$,
the $A$-adic pattern of rectangles 
is said to be 
$A$-{\it adic cyclic
order preserving pattern of rectangles}. 
If in particular one may take $\sigma$ such as 
$\sigma = \id$,
the $A$-adic pattern of rectangles 
is said to be 
$A$-{\it adic order preserving pattern of rectangles}. 

In Definition \ref{defn:PL},
if one may take $\sigma$ such as 
\eqref{eq:sigma}
for some $k \in \{1,\dots,m\}$,
an $A$-adic PL function $f$ is called 
a {\it cyclic order preserving} $A$-{\it adic PL function}.
If in particular,
one may take $\sigma = \id$, 
$f$ is called 
an {\it order preserving} $A$-{\it adic PL function}.

It is easy to see that 
the set $F_A^{\PL}$ of   order preserving $A$-adic PL functions
and 
the set $T_A^{\PL}$ of cyclic order preserving $A$-adic PL functions
 form  subgroups of the group of the $A$-adic PL functions.
Hence we have subgroups of inclusion relations:
\begin{equation*}
F_A^{\PL} \subset 
T_A^{\PL} \subset 
\Gamma_A^{\PL}.
\end{equation*}

The following proposition is immediate 
by definition of order preserving (resp. cyclic order preserving)
$A$-adic PL functions.
\begin{proposition}
An $A$-adic order preserving (resp. cyclic order preserving) 
PL function naturally 
gives rise to
an $A$-adic order preserving (resp. cyclic order preserving) 
pattern of rectangles,
whose rectangle slopes are the slopes of the $A$-adic PL function.
Conversely, 
an $A$-adic order preserving (resp. cyclic order preserving) 
pattern of rectangles 
gives rise to
an $A$-adic order preserving (resp. cyclic order preserving) 
PL function by taking its diagonal lines of the 
corresponding rectangles.
\end{proposition}

In Definition \ref{defn:table},
let 
$
T =\begin{bmatrix}
\mu(1) & \mu(2) & \cdots & \mu(m) \\
\nu(1) & \nu(2) & \cdots & \nu(m) 
\end{bmatrix}
$
be an $A$-adic table 
such that  
$\nu(1) \prec \nu(2) \prec \cdots \prec  \nu(m)$.
Then $T$ is said to be 
\begin{enumerate}
\renewcommand{\theenumi}{\roman{enumi}}
\renewcommand{\labelenumi}{\textup{(\theenumi)}}
\item
{\it order preserving\/} if  
$\mu(1) \prec \mu(2) \prec \cdots \prec  \mu(m)$,
\item
{\it cyclic order preserving\/} if  
$$
\mu(k) \prec \mu(k+1) \prec \cdots \prec  \mu(m)\prec
 \mu(1) \prec \mu(2)   \prec \cdots \prec  \mu(k-1)
$$
for some $k \in \{1,2,\dots,m\}$.
\end{enumerate}
If $T$ is order preserving, it is cyclic order preserving.
These two properties of $A$-adic tables 
are closed under taking expansions of $A$-adic tables 
respectively.
We  see that the set 
$F_A^{\tab}$ 
of the equivalence classes of order preserving $A$-adic tables
and 
the set
$T_A^{\tab}$ 
of the equivalence classes of cyclic order preserving $A$-adic tables 
form subgroups of $\Gamma_A^{\tab}$, respectively. 
Hence we have subgroups of inclusion relations:
\begin{equation*}
F^{\tab}_A \subset T^{\tab}_A \subset \Gamma^{\tab}_A.
\end{equation*}
We further see the following:


\begin{lemma}
For a table $T$, 
let
$f_T$ be the associated $A$-adic PL function.
Then  $T$ is order preserving (resp. cyclic order preserving) 
if and only if the function
$f_{[T]}$ is 
order preserving (resp. cyclic order preserving).
\end{lemma}
We thus have
\begin{proposition}\label{prop:SFTPLorder}
There exist canonical isomorphisms of discrete groups among 
the order preserving (resp. cyclic order preserving)
continuous full group $F_A$ (resp. $T_A$), 
the group $F^{\tab}_A$ (resp. $T^{\tab}_A$)
 of the equivalence classes   
of order preserving (resp. cyclic order preserving) $A$-adic tables  
and 
the group $F^{\PL}_A$ (resp. $T^{\PL}_A$) 
of the  order preserving (resp. cyclic order preserving) 
$A$-adic PL functions on $[0,1)$,
that is 
\begin{equation*}
F_A \cong F^{\tab}_A \cong F^{\PL}_A, 
\qquad
T_A \cong T^{\tab}_A \cong T^{\PL}_A. 
\end{equation*}
\end{proposition}
\begin{proof}
The isomorphisms in
Proposition \ref{prop:tab}
and 
Theorem \ref{thm:SFTPL}
among
$\Gamma_A$,
$\Gamma^{\tab}_A$
and
$\Gamma^{\PL}_A$
 preserve the orders of words, 
so that its restrictions 
yield desired isomorphisms.
\end{proof}

\medskip
In \cite{Brown}, K. S. Brown had extended the Higman--Thomson group $V_N$ to infinite families
$F_{N,r}\subset T_{N,r} \subset V_{N,r}$ for $N=2,3,\dots, r\in \N$
where $V_{N,1} = V_N$ and $F_{N,1} = F_N, T_{N,1} = T_N$. 
Let $A_N$ be the $N \times N$ matrix whose entries are all $1$'s. 
Then our groups $F_{A_N}, T_{A_N},  V_{A_N}$ for the matrix $A_N$ are nothing but
the Brown's triple
$F_{N,1},  T_{N,1},  V_{N,1}$
for $r=1$, respectively.
Let $A_{N, r}$ be the $r\times r$ block matrix whose entries are $N\times N$ matrices
such that
$$
\begin{bmatrix}
0       & \ldots &\ldots  & 0     & A_N \\
1_N     & 0      &\ldots  & \ldots& 0  \\
0       & \ddots &\ddots  &       & \vdots \\
\vdots  & \ddots &1_N    & 0     & 0  \\
0       & \ldots & 0      & 1_N   & 0
\end{bmatrix}
$$
where $1_N$ denotes the identity matrix of size $N$.
Since there exists an isomorphism
from the  Cuntz--Krieger algebra ${\mathcal{O}}_{A_{N,r}}$
for the matrix $A_{N,r}$
to the tensor product ${\mathcal{O}}_{A_N} \otimes M_r({\mathbb{C}})$
such that 
${\mathcal{D}}_{A_{N,r}} = {\mathcal{D}}_{A_N} \otimes D_r$,
where $D_r$ is the commutative $C^*$-algebra of the diagonal elements of the 
$r\times r$ 
full matrix algebra  $M_r({\mathbb{C}})$,
our groups $F_{A_{N,r}}, T_{A_{N,r}},  V_{A_{N,r}}$ for the matrix $A_{N,r}$ 
are nothing but the
Brown's triple
$F_{N,r},  T_{N,r},  V_{N,r}$ (see \cite{MatuiPre2015}, \cite{MatuiPre2016}).
Since $\det(\id - A_{N,r}) = 1 -N$,
the classification of the Higman--Thompson groups 
$V_{N,r}$ corresponds to that of the $C^*$-algebras 
${\mathcal{O}}_N \otimes M_r({\mathbb{C}})$ through Theorem \ref{thm:1.1}
(see \cite[Corollary 6.6]{Ro}, \cite{Pardo}).

\medskip
In \cite{MatuiPre2015}, generalization of higher dimensional analogue of Thomson like groups are studied from the view point of  \'{e}tale groupoids.

\medskip

{\it Acknowledgments:}
The final part of the manuscript was completed 
while KM was visiting the Mittag-Leffler Institute for the program "Classification of operator algebras: complexity, rigidity, and dynamics".
He thanks the institute and the program organizers for the invitation and for their hospitality.  
This work was supported by JSPS KAKENHI grant numbers
22740099,  23540237 and 15K04896.


\end{document}